\documentclass[11pt]{amsart}
\usepackage{}
\usepackage{color}
\usepackage{amssymb}
\usepackage{amsmath}
\usepackage{times}

\newcommand{\s}{\sigma}
\newcommand{\la}{\lambda}
\newcommand{\de}{\delta}

\newcommand{\ov}{\overline}

\newtheorem{theorem}{Theorem}[section]
\newtheorem{lemma}[theorem]{Lemma}

 \theoremstyle{definition}

\theoremstyle{remark}

\numberwithin{equation}{section}

\begin{document}

\title[Complex Equations with Sums of Hessian]
{Second order estimates for equations with sums of Hessian operators on Hermitian manifolds}

\author{Weisong Dong}
\address{School of Mathematics, Tianjin University,
	Tianjin, 300354, China}
\email{dr.dong@tju.edu.cn}

\author{Ruijia Zhang}
\address{Department of Mathematics, Sun Yat-sen University, Guangzhou, 510275, China}
\email{zhangrj76@mail.sysu.edu.cn}


\begin{abstract}
	In this paper, we establish an \emph{a priori} second-order estimate for admissible solutions 
    satisfying a dynamic plurisubharmonic condition
    to equations involving sums of Hessian operators on compact Hermitian manifolds. The estimate is derived using a concavity inequality for complex sum-of-Hessian operators.



\end{abstract}

\maketitle

\section{Introduction}

Let $(M, \omega)$ be a compact Hermitian manifold of complex dimension $n \geq 2$.
For any function $u$ on $M$, let $\chi(z,u)$ be a smooth real (1,1) form on $M$
and let $\psi(z, v, u)$ be a smooth positive function defined on
$\left(T^{1,0}(M)\right)^{*} \times \mathbb{R}$.
Suppose $a$ is a smooth $(1,0)$-form on $M$.
Denote
\begin{equation}
	\label{g}
	g
	=\chi(z, u)+\sqrt{-1} a \wedge \overline{\partial} u - \sqrt{-1} \overline{a} \wedge \partial u+\sqrt{-1} \partial \overline{\partial} u.
\end{equation}
Let $m, k,n \in \mathbb{N}$, $m < k \leq n$ and $b =(b_1, \dots, b_m) \in \mathbb{R}^m$. We consider the following equation involving sums of Hessian operators
\begin{equation}\label{eqn}
F ( \lambda ) := \sigma_k^{(n)} (\lambda) + \sum_{s = 1}^{m} b_s \sigma_{k-s}^{(n)} (\lambda ) = \psi(z, Du, u), \;\; \mbox{for}\;\; 1\leq k \leq n,
\end{equation}
where $\lambda = (\lambda_1, \dots, \lambda_n) $ are the eigenvalues of $g$ with respect to
$\omega$ and $D$ denotes the covariant derivative with respect to $\omega$. 
Here $\sigma_k^{(n)}$ denotes the $k$-th elementary symmetric function on $\mathbb{R}^n$,
with the convention that $\sigma_0^{(n)} = 1$ and $\sigma_k^{(n)} = 0$ for $k < 0$ or $k > n$.

The study of fully nonlinear elliptic equations on complex manifolds has made significant progress in recent years, particularly in connection with complex Hessian equations and their geometric applications. Among these, the complex $k$-Hessian equation (corresponding to $m=0$) plays a central role in the theory of fully nonlinear elliptic equations. When $m=0$, the extremal cases are the Laplace equation ($k=1$) and the complex Monge--Amp\`ere equation ($k=n$, cf. \cite{ChY, TY1, TY2, TW, Yau}), both of which have been extensively studied. 
A typical example in which the left-hand side involves two Hessians is the Fu--Yau equation; see \cite{CHZ, F-Y1, F-Y2, P-P-Z3}. In calibrated geometry, the special Lagrangian equation introduced by Harvey and Lawson \cite{HL} can also be written in the form of equation~\eqref{eqn}. Equations involving sums of Hessians arise naturally in the study of mirror symmetry; see \cite{CJY, CY, LYZ}. Another example comes from the study of the $J$-equation by Collins and Sz\'ekelyhidi \cite{CS}.

The G{\aa}rding cone $\Gamma_{k}^n \subset \mathbb{R}^n$, in which the $k$-Hessian equation is elliptic (cf.~\cite{CNS}), is defined by
\[
	\Gamma_{k}^n :=\left\{\lambda \in \mathbb{R}^{n} : \sigma_{i}^{(n)}(\lambda)>0, i=1, \dots, k \right\}.
\]
Let $A^{1,1}(M)$ be the space of smooth real $(1,1)$-forms on $(M, \omega)$. 
For $h \in A^{1,1}(M)$, 
written in local coordinates as $h = \sqrt{-1} h _{i\bar j} d z^i \wedge d z^{\bar j}$,
define
\[
\sigma_{k}^{(n)} (h) := \sigma_{k}^{(n)} (\lambda({h^i}_j) ),
\]
where ${h^{i}}_{j} = \omega^{i\bar{l}} h_{j\bar{l}}$, and $\lambda({h^{i}}_{j})$ denotes the eigenvalues of the corresponding Hermitian matrix. 
For simplicity, we write $\lambda(h)$ for the eigenvalues of $h$ with respect to the Hermitian metric $\omega$.
For a smooth real $(1,1)$-form $h$, we say $h$ is in the cone $\Gamma_{k}^n (M)$ if $\sigma_{i}^{(n)} (h) > 0$ for $i=1, \ldots, k$.
With the above notation, in local coordinates, \eqref{eqn} can be rewritten as follows:
\begin{equation}\label{eqn2}
	F ( g )= F \left(\omega^{i\bar k} (\chi_{j \overline{k}} + u_{j \overline{k}}
	+ a_{j} u_{\overline{k}} + a_{\overline{k}} u_{j}) \right) = \psi(z, D u, u).
\end{equation}

As in \cite{DXZ}, we consider~\eqref{eqn} under the following Real Root Hypothesis:
\begin{enumerate}
	\item[(RR)] \label{RR}
	The polynomial of degree $m$
	\[
	P(t)=t^m + \sum_{s=1}^m (-1)^s b_s t^{m-s}
	\]
	has $m$ real roots $y_i \in \mathbb{R}$, $i=1,\dots,m$.
\end{enumerate}
Denote $y = (y_1, \dots, y_m)$.
We study the problem within the admissible set of \eqref{eqn}, defined by
\[
\Gamma_k^{(n+m)}=\lbrace\lambda \in \mathbb{R}^{n} : (\la, y)\in \Gamma_k^{n+m}\rbrace.
\]
Clearly, $\Gamma_k^{(n+m)}$ is a convex set, and $F$ is elliptic in this set since 
$\Gamma_k^{(n+m)} \times \{y\} \subset \Gamma_k^{n+m}$. 
Moreover, in the set $\Gamma_k^{(n+m)}$,
\begin{align}\label{sl}
	\sigma_l^{(n+m)}(\lambda, y)
	= \sigma_l^{(n)}(\lambda)
	+ \sum_{s=1}^m b_s \sigma_{l-s}^{(n)}(\lambda)
	> 0, 
	\quad l = 1, \dots, k.
\end{align}
To obtain the desired estimates, we require that the solution-more specifically, $\lambda(g)$-satisfy one of the following conditions:
\begin{enumerate}
	\item \label{con1}
	$\lambda \in \Gamma_k^{(n+m)}$ and $y_i \geq 0$ for all $i = 1, \dots, m$;
	
	\item \label{con2}
	$\lambda \in \Gamma_{k-1}^n$ and $b_i \geq 0$ for all $i = 1, \dots, m$.
\end{enumerate}
We note that Condition~\eqref{con2} implies 
$\lambda \in \Gamma_k^{(n+m)}$ by \eqref{eqn} and \eqref{sl}, 
while Condition~\eqref{con1} implies $b_i \ge 0$, 
but does not, in general, guarantee that $\lambda \in \Gamma_{k-1}^n$.

The complex $k$-Hessian equation corresponding to $b_s = 0$, $s = 1, \dots, m$, is given by
\begin{equation}\label{Sk2}
	\sigma_{k}^{(n)}(g)
	= \psi(z, D u, u).
\end{equation}
This equation has been studied extensively in the literature.
The complex $k$-Hessian equation without gradient terms on either side was solved on closed Hermitian manifolds by Sz\'ekelyhidi \cite{G.S} and Zhang \cite{ZhangDK}, respectively. 
Recently, Guo and Song \cite{GS} found a necessary and sufficient condition
for the existence of a solution to Hessian equations on closed manifolds.
Prior to that, some necessary and sufficient conditions for parabolic Hessian quotient equations on Kähler manifolds were established, arising from the study of the $J$-flow; see \cite{FLM, SW}.
The Dirichlet problem for the complex $k$-Hessian equation was studied by Li \cite{LiSY} and Collins-Picard \cite{CP}; for the case of complex Hessian quotient equations, see \cite{GL, GuanSun}.
For complex $k$-Hessian equations \eqref{Sk2},
Phong, Picard, and Zhang \cite{PPZ} first established second order estimates
for solutions in $\Gamma_n^n (M)$ on K\"ahler manifolds.
On Hermitian manifolds, Dong and Li \cite{DL} proved the estimate
for $g\in \Gamma_n^n (M)$.
For $g\in \Gamma_{k+1}^n (M)$, second order estimates were obtained in \cite{Dong}.
For equations involving sums of Hessians, the second order estimates was established by Dong \cite{Dong1} for $g \in \Gamma_{k+1}^n (M)$.

In this paper, we extend the above-mentioned second order estimates to a broader setting, which can be justified by Lemma \ref{sigmak} (8).
Let $\lambda^{\downarrow} \in \mathbb{R}^n$ denote the decreasing rearrangement of $\lambda \in \mathbb{R}^n$, defined by 
\[
\lambda^{\downarrow} : = (\lambda_1, \ldots, \lambda_n),\; \mbox{where}\;
\lambda_1 \geq \lambda_2 \geq  \cdots \geq \lambda_n.
\]
For $\delta > 0$, suppose that $\lambda^{\downarrow} (g)$ satisfies
\begin{equation}
	\label{g-cone}
 \lambda_n > -\delta \lambda_1. 
\end{equation}
In the real setting, this can be viewed as a semiconvexity condition with a variable lower bound on the Hessian, known in \cite{SY} as a dynamic semiconvexity condition.
In the complex setting, we introduce the corresponding notion, which we refer to as a \textit{dynamic plurisubharmonic} condition.
We now state the following estimate.
\begin{theorem}\label{thm1}
	Let $(M,\omega)$ be a compact Hermitian manifold of complex dimension $n \geq 2$, and let $u \in C^{\infty}(M)$. 
	
	Assume that \eqref{eqn} satisfies Hypothesis~(RR) and that
	\[
	\chi(z,u) \geq \varepsilon \omega
	\]
	for some constant $\varepsilon > 0$. Assume that the smooth real $(1,1)$-form $g$, defined in \eqref{g}, satisfies Condition \eqref{con1} (or Condition \eqref{con2}) and \eqref{g-cone} for some sufficiently small $\delta > 0$. 
	
	Then we obtain the following uniform second-order estimate:
	\begin{equation}
		|D \overline{D} u|_{\omega} \leq C,
	\end{equation}
	where $C$ is a uniform constant depending only on $(M,\omega)$, $\varepsilon$, $\delta$, $n$, $k$, $m$, $\chi$, $\psi$, $a$, $\|b\|$, and $\|u\|_{C^1}$.
\end{theorem}

The real $(1,1)$-form $g$, which contains gradient terms,
is a typical structure arising in the study of Gauduchon conjecture.
The conjecture was resolved by Sz\'ekelyhidi, Tosatti, and Weinkove
\cite{STW}, who established a second order estimate of the form 
$\max_M |D\overline Du| \leq C (1 + \max_M |Du|^2)$, where $C$ is a uniform constant depending only on the known data, but not on $u$. Related results can also be found in the work of Guan and Nie \cite{GN}. There, the authors also derived interior second order estimates for the equation under stronger assumptions on $\psi$ and $g$ with respect to $Du$. 
Tosatti and Weinkove \cite{T-W19} extended the study of the Monge-Amp\`ere equation with gradient terms in $g$ on Hermitian manifolds from complex dimension 2, as studied by Yuan \cite{Yuan}, to higher dimensions.
Second order estimates with the particular dependence on $|Du|$ were first
derived by Hou, Ma, and Wu \cite{HMW} on K\"ahler manifolds for $k$-Hessian equations. 
Sz\'ekelyhidi \cite{G.S} generalized the estimates to fully nonlinear equations on Hermitian manifolds.
These estimates played a key role in enabling the use of blow-up analysis and Liouville-type theorems to derive gradient estimates, as shown by Dinew and Ko\l odziej \cite{DK}. 
Another example in which $g$ depends on gradient terms arises in the study of the Calabi-Yau equation; see \cite{FinoLi}.

The key contributions of this article are that we do not require $\psi$ to be convex with respect to $Du$, and our real $(1,1)$-form $g$ does not satisfy the concavity condition with respect to $Du$ imposed in \cite{GN}, thereby extending the scope of previous results.
Hence, we need to take care of the negative terms arising from the differentiation of $\psi$.
To this end, we establish a key concavity inequality (see Lemma \ref{Fconcavity}) to control the negative third-order terms that arise from the fully nonlinear operator.
This inequality can also be employed to derive curvature estimates for hypersurfaces in space forms. 
Establishing a concavity inequality has been central in this line of research. Guan, Ren, and Wang \cite{GRW} achieved this for the real
$k$-Hessian equation by assuming that the solution lies in the $\Gamma_n^n$ cone. 
See the work of Chu \cite{Chu} for a simple proof.
Later, Li, Ren and Wang \cite{LRW} extended the result to solutions in $\Gamma_{k+1}^n$ cone. 
The problem was completely solved for solutions in the $\Gamma_k^n$ cone in the two cases $k = n-1$ and $k = n-2$ by Ren and Wang \cite{RW1, RW2}.
For simplified proofs, see the work of Lu and Tsai \cite{LT}.
However, their method cannot be applied to complex $k$-Hessian equations due to the presence of complex conjugate terms.
When $k=2$, the operator has a special structure that enables the derivation of second order estimates; see Chu, Huang and Zhu \cite{CHZ3}.
Therefore, it remains an open problem to establish second order estimates for solutions of the $k$-Hessian equation \eqref{Sk2} in the $\Gamma_k^n$ cone for $2 < k < n$.

The structure of the paper is as follows.
In Section 2, we introduce some useful notation, properties of the $k$-th elementary symmetric function,
and some preliminary calculations and estimates.
In Section 3, we prove the key inequality for the $k$-th elementary symmetric function.
In Sections 4, we prove our main Theorem \ref{thm1}.

\textbf{Acknowledgements}:
This work was carried out while the first-named author was visiting the University of Granada. He thanks Prof. José Gálvez for his support and gratefully acknowledges IMAG and the Department of Geometry and Topology for their warm hospitality. His visit was supported by the China Scholarship Council (No. 202406250051).

\section{Preliminaries}

For $x = (x_1, \ldots, x_n) \in \mathbb{R}^n$, we denote by $s_{k;i}(x)$ the $k$-th elementary symmetric function evaluated at $x$ with $x_i = 0$, and by $s_{k;ij}(x)$ the $k$-th elementary symmetric function evaluated at $x$ with $x_i = x_j = 0$. 

We define
\[
x \mid i := (x_1, \ldots, x_{i-1}, 0, x_{i+1}, \ldots, x_n).
\]

The following properties of the $k$-th elementary symmetric function are well known.

\begin{lemma}
	\label{sigmak}
	Let $x = (x_1, \ldots, x_n) \in \mathbb{R}^n$ and $k = 1, \ldots, n$. Then the following hold:
	\begin{itemize}
		\item[(1)] $s_k = s_{k;i} + x_i s_{k-1;i}$ for all $1 \leq i \leq n$;
		
		\item[(2)] $\displaystyle \sum_{i=1}^n s_{k;i} = (n-k) s_k$ and 
		$\displaystyle \sum_{i=1}^n s_{k-1;i} x_i = k s_k$;
		
		\item[(3)] $\displaystyle \sum_{i=1}^n s_{k-1;i} x_i^2 
		= s_1 s_k - (k+1) s_{k+1}$;
		
		\item[(4)] If $x \in \Gamma_k^n$ with $x_1 \geq \cdots \geq x_n$, then
		$x_1 s_{k-1;1} \geq \frac{k}{n} s_k$;
		
		\item[(5)] If $x \in \Gamma_k^n$, then $x \mid i \in \Gamma_{k-1}^n$ for all $1 \leq i \leq n$;
		
		\item[(6)] If $x \in \Gamma_k^n$ and $x_i \geq x_j$, then
		$s_{k-1;i} \leq s_{k-1;j}$;
		
		\item[(7)] If $x \in \Gamma_k^n$ with $x_1 \geq \cdots \geq x_n$, then
		\[s_k \leq C\, x_1 \cdots x_k \;\mbox{and}\; 
		s_{k-1} \geq x_1 \cdots x_{k-1};
		\]
		
		\item[(8)] If $x \in \Gamma_k^n$, and $s_k \leq A_1$ and $s_{k+1} \geq -A_2$
		for some constants $A_1, A_2 \geq 0$, then for all $i$,
		\[
		x_i + K \geq 0,
		\]
		where $K$ is a uniform positive constant depending only on $n$, $k$, $A_1$, and $A_2$.
	\end{itemize}
\end{lemma}

\begin{proof}
For (1)--(3), the identities are straightforward. 
For (4), see \cite{HMW, LTr}. 
For (5), see \cite{HS}. It can also be derived from the following alternative characterization of $\Gamma_k$ (see \cite{K}):
\[
\Gamma_k^n
=
\left\{
x \in \mathbb{R}^n :
\partial^{\alpha} s_k(x) > 0
\ \text{for all multi-indices } \alpha
\text{ with } 0 \le |\alpha| \le k
\right\}.
\]
For (6) and (7), see \cite{YYLi, Wang}. 
For (8), see \cite{LRW, Zhang}.
\end{proof}

We recall some previous lemmas from \cite{HS, Zhang}.
\begin{lemma}\label{xk}
	Suppose $x^{\downarrow} \in \Gamma_k^n$. Then, we have
	\[
	x_k \leq C \left(s_k^{\frac{1}{k}}+|x_n|\right)
	\]
	for some constant $C$ depending on $n$ and $k$.
\end{lemma}

\begin{proof}
	
Since $x_k + \cdots + x_n > 0$, we have $|x_n| \leq (n-k) x_k$, and hence
\begin{equation}
\label{c}
  x_1 \cdots x_{k-1} x_k \leq s_k + C (n, k) x_1 \cdots x_{k-1} |x_n|.  
\end{equation}
	Therefore, we obtain
	\begin{equation}
		\label{c1}
		 x_k \leq 2C(n,k) \left(s_k^{\frac{1}{k}}+|x_n|\right).
	\end{equation}
	This completes the proof of the lemma.
\end{proof}
Set
\[
q_k = \frac{s_k}{s_{k-1}}\; \mbox{and} \; q_{k;i} = \frac{s_{k;i}}{s_{k-1;i}}.
\] 
Note that $q_1 = s_1$ since $s_0 \equiv 1$.
The following lemma was established in the real setting in \cite{HS}. In this paper, we present a proof in the complex case for completeness.
\begin{lemma}\label{q2}
	Given $k\geq 2$, if $x \in \Gamma_k^n$, then for any $\xi \in \mathbb{C}^n$, we have
	\[
	- \partial_{\ov \xi} \partial_\xi q_2 = \sum_{i=1}^n \frac{\Big| \xi_i - \frac{x_i}{s_1} s_1 (\xi) \Big|^2}{s_1}
	\]
	and 
	\[
	\partial_{\ov \xi} \partial_\xi q_{k+1} \leq \sum_{i=1}^n \frac{x_i^2 \partial_{\ov \xi} \partial_\xi q_{k;i}}{(k+1)(q_{k;i} + x_i)^2}.
	\]
\end{lemma}

\begin{proof}
By a direct calculation, we obtain that
\[\begin{aligned}
&\; - \partial_{\ov \xi} \partial_\xi q_2 \\
= &\; - \frac{\sum_{i \neq j} \xi_i \xi_{\ov j} }{s_1} + \frac{\sum_{ij} s_{1;i} \xi_i \xi_{\ov j}}{s_1^2} + 
 \frac{ \sum_{ij} s_{1;j} \xi_i \xi_{\ov j}}{s_1^2} - 2  \frac{\sum_{ij} s_2 \xi_i \xi_{\ov j}}{s_1^3}\\
 = &\; - \frac{ |s_1 (\xi)|^2 - |\xi|^2 }{s_1} + \frac{ (s_1 - x_i) \xi_i s_1 (\ov \xi)}{s_1^2} + \frac{ (s_1 - x_j)\xi_{\ov  j} s_1 (\xi)}{s_1^2} - \frac{s_1^2 - |x|^2}{s_1^3} |s_1(\xi)|^2\\
 = &\; \frac{|\xi|^2}{s_1} - \frac{x \cdot \xi s_1 (\ov \xi)}{s_1^2}
 - \frac{x \cdot \ov \xi s_1 ( \xi)}{s_1^2} + \frac{|x|^2 |s_1 (\xi)|^2}{s_1^3}.
\end{aligned}\]
Thus, the equation is proved.
For the inequality, recall formula (2.12) in \cite{HS} that
\[
(k+1) q_{k+1} (x) = \sum_{i=1}^n \Big(x_i - \frac{x_i^2}{q_{k;i} (x) + x_i}\Big).
\]
Therefore, we have
\[
\begin{aligned}
   (k+1) \partial_{\ov\xi} \partial_{\xi} q_{k+1} 
   = \sum_{i=1}^n \frac{x_i^2 \partial_{\ov \xi} \partial_\xi q_{k;i}}{(q_{k;i} + x_i)^2} + \sum_{i=1}^n P_i,
\end{aligned}
\]
where
\[
\begin{aligned}
P_i
= &\; - \frac{2 \xi_i \xi_{\ov i}}{q_{k;i} + x_i} + \frac{2 x_i \xi_i \xi_{\ov i}}{(q_{k;i} + x_i)^2} + \frac{2 x_i \xi_i \partial_{\ov\xi} q_{k;i}}{(q_{k;i} + x_i)^2}\\
   &\; + \frac{2 x_i\xi_i \xi_{\ov i}}{(q_{k;i} + x_i)^2} - \frac{2 x_i^2 \xi_i \xi_{\ov i}}{(q_{k;i} + x_i)^3} - \frac{2 x_i^2 \xi_i \partial_{\ov \xi} q_{k;i} }{(q_{k;i} + x_i)^3}\\
   &\; + \frac{2 x_i \xi_{\ov i} \partial_{ \xi} q_{k;i} }{(q_{k;i} + x_i)^2} - \frac{2 x_i^2 \xi_{\ov i} \partial_{\xi} q_{k;i} }{(q_{k;i} + x_i)^3} - \frac{2 x_i^2 \partial_{\ov \xi} q_{k;i} \partial_{\xi} q_{k;i} }{(q_{k;i} + x_i)^3}.
\end{aligned}
\]
A direct computation gives that
\[
- \frac{2 \xi_i \xi_{\ov i}}{q_{k;i} + x_i} + \frac{4 x_i \xi_i \xi_{\ov i}}{(q_{k;i} + x_i)^2}  - \frac{2 x_i^2 \xi_i \xi_{\ov i}}{(q_{k;i} + x_i)^3} = - \frac{ 2 q_{k;i}^2 \xi_i \xi_{\ov i}}{(q_{k;i} + x_i)^3},
\]
and
\[
 \frac{2 x_i \xi_i \partial_{\ov\xi} q_{k;i}}{(q_{k;i} + x_i)^2} - \frac{2 x_i^2 \xi_i \partial_{\ov \xi} q_{k;i} }{(q_{k;i} + x_i)^3}
= \frac{2 x_i q_{k;i} \xi_i \partial_{\ov \xi} q_{k;i} }{(q_{k;i} + x_i)^3}.
\]
Note that, for $x \in \Gamma_k^n$,
\[
q_{k;i} + x_i = \frac{s_k}{s_{k-1;i}} > 0.
\]

Hence, we have
\[\begin{aligned}
     P_i
    = &\; \frac{2}{(q_{k;i} + x_i)^3} (-q_{k;i}^2 |\xi_i|^2 + x_i q_{k;i} \xi_i \partial_{\ov\xi} q_{k;i} + x_i q_{k;i} \xi_{\ov i} \partial_{\xi} q_{k;i} - x_i^2 \partial_{\ov\xi} q_{k;i}\partial_{\xi} q_{k;i} )\\
    = &\; - \frac{2}{(q_{k;i} + x_i)^3} (x_i \partial_{\xi} q_{k;i} - q_{k;i} \xi_i)(x_i \partial_{\ov \xi} q_{k;i} - q_{k;i} \xi_{\ov i})\leq 0.
\end{aligned}
\]
The proof of the lemma is completed.
\end{proof}

Denote by $[\xi]_{i_1, \ldots, i_{\ell}}$ the vector obtained from $\xi$ by setting its $i_1$-th, $\ldots$, $i_{\ell}$-th components to zero.
By the above lemma, we can further derive that
\begin{equation}
\label{2D-q_2}
\begin{aligned}
    - \partial_{\ov \xi} \partial_{\xi} q_{2;i_1, \ldots, i_\ell} 
    = \sum_{i\notin \{i_1, \ldots, i_\ell\} } 
    \frac{|\xi_i - \frac{x_i}{s_{1;i_1, \ldots, i_\ell }} s_{1;i_1, \ldots, i_\ell } (\xi)|^2}{s_{1;i_1, \ldots, i_\ell }}
    \geq \frac{|[\xi]_{i_1, \ldots, i_\ell}^{\perp}|^2}{s_{1;i_1, \ldots, i_\ell }},
\end{aligned}
\end{equation}
where $[\xi]_{i_1, \ldots, i_\ell} = c [x]_{i_1, \ldots, i_\ell} + [\xi]_{i_1, \ldots, i_\ell}^\perp$ for some $c \in \mathbb{C}$, and that
\begin{equation}
\label{2D-q}
    \partial_{\ov \xi} \partial_{\xi} q_{k+1; i_1, \ldots, i_\ell}
    \leq \sum_{i\notin \{i_1, \ldots, i_\ell\} } 
    \frac{x_i^2 \partial_{\ov \xi} \partial_{\xi} q_{k;i_1, \ldots, i_\ell,i}}{(k+1) (q_{k; i_1, \ldots, i_\ell, i} + x_i)^2}.
\end{equation}
Since
\[
q_{k; i_1, \ldots, i_\ell, i} + x_i = \frac{s_{k;i_1, \ldots, i_\ell}}{s_{k-1; i_1, \ldots, i_\ell, i}},
\]
we obtain from \eqref{2D-q} that
\[
- s_{k;i_1, \ldots, i_\ell} \partial_{\ov \xi} \partial_{\xi} q_{k+1; i_1, \ldots, i_\ell}
    \geq - \sum_{i\notin \{i_1, \ldots, i_\ell\} } 
    \frac{ x_i^2 \partial_{\ov \xi} \partial_{\xi} q_{k;i_1, \ldots, i_\ell,i} s_{k-1; i_1, \ldots, i_\ell, i}}{(k+1) \frac{s_{k;i_1, \ldots, i_\ell}}{s_{k-1; i_1, \ldots, i_\ell, i}}}.
\]
Using the above inequality repeatedly, we arrive at
\[
\begin{aligned}
    &\; - s_{k-1} \partial_{\ov \xi} \partial_{\xi} q_{k}\\
    \geq &\; - \sum_{i_1} \frac{x_{i_1}^2 \partial_{\ov\xi} \partial_{\xi} q_{k-1;i_1} s_{k-2;i_1}}{k \frac{s_{k-1}}{s_{k-2;i_1}}}\\
    \geq &\; \cdots \\
    \geq &\; - \sum_{i_1} \; \cdots \sum_{i_{k-2} \notin \{i_1, \ldots, i_{k-3}\}} \frac{2 x_{i_1}^2 \cdots x_{i_{k-2}}^2 \partial_{\ov\xi} \partial_{\xi} q_{2; i_1, \ldots, i_{k-2}} s_{1; i_1, \ldots, i_{k-2}}}{ k! \frac{s_{k-1}}{s_{1; i_1, \ldots, i_{k-2}} } }\\
    \geq &\; \sum_{i_1} \; \cdots \sum_{i_{k-2} \notin \{i_1, \ldots, i_{k-3}\}} \frac{2 x_{i_1}^2 \cdots x_{i_{k-2}}^2  |[\xi]^\perp_{i_1, \ldots, i_{k-2}}|^2 }{ k! \frac{s_{k-1}}{s_{1; i_1, \ldots, i_{k-2}} } },
\end{aligned}
\]
where we used \eqref{2D-q_2} in the last inequality.
Since $s_{k-1}\leq C x_1 \cdots x_{k-1}$ and $s_{1; i_1, \ldots, i_{k-2}} \geq x_j$, where $j$ is the remaining index in $\{1, \ldots, k-1\}$ not among $\{i_1, \ldots, i_{k-2}\}$,
we obtain from above that
\begin{equation}
\label{2D-q_k}
- s_{k-1} \partial_{\ov \xi} \partial_{\xi} q_{k}
\geq C\sum_{j=1}^{k-1} x_1 \cdots \hat x_j \cdots x_{k-1} |[\xi]^\perp_{1, \ldots, \hat j, \cdots, k-1 }|^2. 
\end{equation}
Now, as in \cite{Zhang}, we can show 
\begin{lemma}
\label{qk-1}
Let $x \in \Gamma_k$.
Assume that $\zeta = (0, \zeta_2, \ldots, \zeta_n) \in \mathbb{C}^n$ satisfies
\[
(\zeta_k, \ldots, \zeta_n) \perp (x_k, \ldots, x_n).
\]
Then,
\[
- s_{k-1} \partial_{\ov \zeta} \partial_\zeta q_{k} \geq C x_1 \cdots x_{k}^2 \sum_{j=2}^{k-1} \frac{|\zeta_j|^2}{x_j^3} + C x_1 \cdots x_{k-2} \sum_{p=k}^n |\zeta_p|^2,
\]
where $C$ depends only on $n$ and $k$.
\end{lemma}
\begin{proof}
    Let \[
\alpha_i = (x_i, \beta') = (x_i, x_k, \ldots, x_n)
    \]
and
\[
\beta_i = (|\beta'|, - \frac{x_i}{|\beta'|} \beta').
\]
We see that $\alpha_i \perp \beta_i$.
Therefore, we have
\[
(\zeta_i, \zeta_k, \ldots, \zeta_n) = (0, \zeta_k, \ldots, \zeta_n) + 
\frac{\zeta_i |\beta'|}{x_i^2 + |\beta'|^2} \beta_i + 
\frac{\zeta_i x_i }{x_i^2 + |\beta'|^2} \alpha_i.
\]
Substituting the above decomposition into \eqref{2D-q_k}, we obtain that
\[
\begin{aligned}
    - s_{k-1} \partial_{\ov \zeta} \partial_{\zeta} q_{k}
\geq &\; C\sum_{j=2}^{k-1} x_1 \cdots \hat x_j \cdots x_{k-1}
\left(\sum_{p=k}^{n} |\zeta_p|^2 + \frac{|\zeta_j|^2 |\beta'|^2}{x_j^2 + |\beta'|^2}\right)\\
\geq &\;  C\sum_{j=2}^{k-1} x_1 \cdots \hat x_j \cdots x_{k-1}
\left(\sum_{p=k}^{n} |\zeta_p|^2 + C \frac{x_k^2}{x_j^2} |\zeta_j|^2\right),
\end{aligned}
\]
which finishes the proof of the lemma.
\end{proof}

Let
 \[
F_i (\la)= \frac{\partial F (\la)}{\partial \la_i}\;
\mbox{and}\;
F_{ij} (\la) = \frac{\partial^2 F (\la)}{\partial \la_i \partial \la_j}.
\]
As in \cite{CP} and \cite{P-P-Z3}, we define the tensor
\[
\sigma_k^{p\bar q} = \frac{\partial \sigma_k}{\partial {g^r}_p} \omega^{r \bar q}\; \mbox{and}\;
\sigma_k^{p\bar q, r\bar s} = \frac{\partial^2 \sigma_k}{\partial {g^a}_p \partial {g^b}_r} \omega^{a\bar q} \omega^{b\bar s}.
\]
Then, for equation \eqref{eqn}, we can similarly define $F^{p\bar q}$,
$F^{p \overline{q}, r \overline{s}}$ and
\[
\mathcal{F}=\sum_{p} F^{p \overline{q}} \omega_{\bar q p}.
\]
At a diagonal matrix $g$ with distinct eigenvalues, we have (see \cite{J.M.B}),
\begin{equation}
	\label{symmetric func 1th deriv} 
    F^{i\ov j} =  \delta_{ij} F_i
\end{equation}
and
\begin{equation}
	\label{symmetric func 2th deriv} 
    F^{i\ov j, r\ov s} \eta_{i\ov j } \eta_{r\ov s } = \sum F_{ij} \eta_{i\ov i } \eta_{j \ov j } + \sum_{p\neq q}\frac{F_p - F_q}{\lambda_p-\lambda_q} | \eta_{p\ov q}|^2,
 \end{equation}
where $\eta_{i\ov j}$ is an arbitrary Hermitian matrix.
Note that these formulas remain well defined even when the eigenvalues are not distinct since the expressions can be interpreted in the sense of limits.

In local complex coordinates $\left(z_{1}, \ldots, z_{n}\right)$, the subscripts of a function $u$ always denote the covariant derivatives of $u$ with respect to the Hermitian metric $\omega$, taken in the directions of the local frame $\left\{\frac{\partial}{\partial z^{1}}, \ldots, \frac{\partial}{\partial z^{n}}\right\}$. That is, we write
\[
u_{i} = D_{\frac{\partial}{\partial z^i}} u, \; 
u_{i \ov{j}} = D_{\frac{\partial}{\partial \ov{z}^{j}} } D_{\frac{\partial}{\partial z^i}} u, \; 
u_{i \ov{j} l} = D_{\frac{\partial}{\partial z^l}} D_{\frac{\partial}{\partial \ov{z}^{j}}} D_{\frac{\partial}{ \partial zi}} u,
\]
where $D$ denotes the Chern connection associated with $\omega$.
We recall the following commutation formulas on Hermitian manifolds (see \cite{T-W3} for more details):
\begin{equation}\label{order}
\begin{aligned}
u_{i \ov{j} \ell} = &\; u_{i \ell \ov{j}} - u_{p} R_{\ell \ov{j} i}{}^p,\;
u_{p \ov{j} \ov{m}} =  u_{p \ov{m} \ov{j}} - \ov{T_{m j}^{q}} u_{p \ov{q}},\;
u_{i \ov{q} \ell} = u_{\ell \ov q i} - T_{\ell i}^{p} u_{p \ov{q}},\\
u_{i \ov{j} \ell \ov{m}} = &\; u_{\ell \ov{m} i \ov{j}}
+ u_{p \ov{j}} R_{\ell \ov{m} i}{}^{p}
- u_{p \ov{m}} R_{i \ov{j} \ell}{}^{p} - T_{\ell i}^{p} u_{p \ov{m} \ov{j}} - \ov{T_{m j}^{q}} u_{\ell \ov{q} i}
- T_{i \ell}^{p} \ov{T_{m j}^{q}} u_{p \ov{q}}.
\end{aligned}
\end{equation}

As usual, we define
\begin{equation}
	\begin{aligned}
		|D D u|_{F \omega}^{2} = F^{p \overline{q}} \omega^{m \overline{\ell}} u_{mp} u_{\ov \ell \ov q},\;
		|D \ov{D} u|_{F \omega}^{2} = F^{p \overline{q}} \omega^{m \overline{\ell}} u_{p\ov\ell} u_{m\ov q},
	\end{aligned}
\end{equation}
and, for any (1,0)-form $\eta = \eta_p dz^p$,
\begin{equation}
	|\eta|_{F}^{2} = F^{p \overline{q}} \eta_{p} \eta_{\overline{q}},
\end{equation} 
where $\{\omega^{m\bar l}\}$ denotes the inverse of the Hermitian metric $\omega$.

Next, we recall some basic calculations from \cite{Dong1} that will be used in the proof of Theorem \ref{thm1}. Throughout the following computations, we denote by 
$C$ a uniform constant depending only on the known data as specified in Theorem \ref{thm1}; the value of $C$ may vary from line to line.
The calculations are performed at a fixed point $p \in M$, using local complex coordinates centered at $p$, such that the Hermitian metric satisfies
\[
\omega=\sqrt{-1} \sum \delta_{k \ell} dz^{k} \wedge d \overline{z}^{\ell},
\]
and the matrix $g_{i \ov{j}}$ is diagonal at $p$.
Denote the eigenvalues of $g_{i \overline{j}}$ at this point by $\lambda_{1}, \lambda_{2}, \dots, \lambda_{n}$, arranged in descending order 
\[
\lambda_{1} \geq \lambda_{2} \geq \cdots \geq \lambda_{n}.
\]
Note that under this coordinate, the matrix $\{F^{p\ov q}\}$ is also diagonal as $\{ g_{p\ov q } \}$ is diagonal.
Differentiating equation \eqref{Sk2}, we obtain
\begin{equation}\label{differential equ}
	F^{p \overline{q}} D_{i} g_{p \overline{q}} = \psi_{z_i} + \psi_u u_i + \psi_{v_\ell} u_{\ell i} + \psi_{\ov v_\ell} u_{\ov \ell i}.
\end{equation}
Differentiating this equation a second time at the fixed point, we get
\[
	\begin{split}
		&\; F^{p \overline{q}} D_{\ov{j}} D_{i} g_{p \ov q} + F^{p \overline{q}, r \overline{s}}D_{\overline{j}} g_{r \overline{s}} D_{i} g_{p \overline{q}} = D_{\ov j}D_{i} \psi \\
		\geq & -C(|DDu|^{2} + |D \ov{D} u|^{2}) + \sum_\ell \psi_{v_\ell} g_{i \ov{j} \ell} + \sum_\ell \psi_{\ov{v}_\ell} g_{i \ov{j} \ov{\ell}} - C \lambda_{1}
	\end{split}
\]
for $\lambda_1$ sufficiently large.
By \eqref{order}, the Cauchy-Schwarz inequality,
and the above inequality, we have
\begin{equation}
	\begin{aligned}\label{dif-eqn2}
		&\; F^{p \overline{q}} D_{\overline{q}} D_{p} g_{i \overline{j}}\\
		\geq &\ -F^{p \overline{q}, r \overline{s}} D_{\ov{j}}  g_{r \overline{s}} D_{i} g_{p \overline{q}} + \sum_\ell ( \psi_{v_\ell}
		g_{i \ov{j} \ell} + \psi_{\ov{v}_\ell} g_{i \ov{j} \ov{\ell}} )\\
		&\ - F^{p \ov{q}} \big( T_{p i}^{a} u_{a \overline{q} \overline{j}} 
	    + \overline{T_{q j}^{a}} u_{p \overline{a} i} \big) 
	    - C( |DDu|^{2} + |D \ov{D}u|^{2} + {\lambda}_{1} \mathcal{F} + \lambda_{1})\\
		&\ + F^{p \overline{q}}(a_i u_{\ov j p \ov q}+a_{\ov j} u_{i p \ov q} - a_p u_{\ov q i \ov j} - a_{\ov q} u_{pi \ov j})\\
		&\ - \frac{1}{4}|D \ov{D} u|_{F \omega}^{2} - \frac{1}{4}|DDu|_{F \omega}^{2} - C\mathcal{F},
	\end{aligned}
\end{equation}
where $T$ denotes the torsion tensor associated with the Chern connection.
See inequalities (17) and (18) in \cite{Dong1}.
By direct calculation
and the differential equation \eqref{differential equ}, we obtain the following,
\begin{equation}
	\begin{aligned}\label{derivative of Du'}
		F^{p \overline{q}} ( |D u|^{2} )_{p \overline{q}}
		\geq &\; 2 \operatorname{Re}
		\big\{\sum_{p, m} \left( u_{mp} u_{\overline{p}} + u_{p} u_{m \overline{p}} \right) \psi_{v_{m}}\big\}\\
		&\; + \frac{1}{2}|D D u|_{ F \omega}^{2} + \frac{1}{2}|D \ov{D} u|_{ F \omega}^{2} - C \mathcal{F} - C.
	\end{aligned}
\end{equation}
It is also easy to see that, under Condition \eqref{con1} or \eqref{con2},
\begin{equation}
	\label{D-u}
	\begin{aligned}
		- F^{p\bar q } u_{p\bar q} 
		\geq \varepsilon \mathcal{F} + F^{p\bar q} ( a_p u_{\bar q} + a_{\bar q} u_p ) - k\psi,
	\end{aligned}
\end{equation}
since $\chi \geq \varepsilon \omega$.
For more details about the above calculations, we refer the reader to \cite{Dong1}.

\section{A Concavity Lemma}
In the following, we establish the crucial concavity inequality for the complex sum-of-Hessian operator $F$.
\begin{lemma}
	\label{Fconcavity}
	Suppose $(\la,y) \in \Gamma_k^{n+m}$ and $\la_1\geq \la_2\geq \cdots\geq \la_n \geq -\delta \lambda_1$. Then,  given $\gamma \in (0, 1)$, for sufficiently small $\delta$ and sufficiently large $\lambda_1$ depending only on $n$, $k$, $m$, and $\s_k^{(n+m)}(\la,y)$, the following inequality holds at $\lambda^{\downarrow}:$
	\begin{align}\label{key F}
		-\sum_{\substack{p \neq q}} \frac{F_{pq} \xi_p \xi_q}{F} + 2\frac{\left(\sum_{i=1}^{n} F_{i} \xi_i\right)^2}{F^2} + \sum_{i>1} \frac{F_{i}\xi_i^2}{(1+\delta)\lambda_1 F} \geq (1-\gamma)\frac{F_{1} \xi_1^2}{\lambda_1 F}   
	\end{align}
	for any vector $\xi = (\xi_1, \ldots, \xi_n) \in \mathbb{C}^{n}$.
\end{lemma}
\begin{proof}
Define
\[
	\hat{\la} = (\la_1,\ldots,\la_n,y_1,\dots,y_m)\; \mbox{and}\;
	\hat{\xi}= (\xi_1,\ldots,\xi_n,0,\ldots,0).\]
Then
\[
F(\lambda) = \sigma_k^{(n+m)} (\hat \lambda).
\]
Next, we rearrange $\hat{\lambda}$ and $\hat{\xi}$ accordingly:
\begin{align*}
	\bar{\lambda}=\hat{\lambda}^{\downarrow}
	= (\lambda_1,\hat \lambda_{\tau(2)},\ldots, \hat \lambda_{\tau(n+m)})\; \mbox{and}\;
	\bar{\xi}= (\xi_1, \hat \xi_{\tau(2)},\ldots, \hat \xi_{\tau(n+m)}),
\end{align*}
where $\tau$ is a permutation.
Since $\la_i = \hat \la_i$ for $1\leq i \leq n$, we have
\[
F_i(\la) = \sigma_{k-1;i}^{(n)} (\la)+\sum_{r=1}^m a_r \sigma_{k-r-1;i}^{(n)} (\la)
= \sigma_{k-1;i}^{(n+m)} (\hat{\la})
\]
and
\[
F_{ij} (\la) = \sigma_{k-2;ij}^{(n+m)} (\hat{\la})\; \mbox{for}\; 1 \leq i\neq j \leq n.
\]
By applying the following lemma to $\s_k^{(n+m)}(\bar \la)$ and $\bar \xi$,	
we obtain inequality \eqref{key F}.
\end{proof}

We now establish the following concavity inequality for the complex $k$-Hessian operator. The proof sketch follows the arguments in \cite{DXZ}, Lemma 3.1 (see also \cite{Zhang}, Lemma 1.1). The complex structure leads to differences in several cases. Based on the concavity of the complex Hessian quotient operator $q_k$ derived in Lemma \ref{qk-1}, we provide a proof below.
\begin{lemma}
	\label{key-c}
    Suppose $\lambda \in \Gamma_k^n$, $\la_1\geq\la_2\geq\cdots\geq \la_n$ and $\lambda_n \geq -\delta \lambda_1$ for some constant $\delta \in (0, 1)$. 
	Then, given $\gamma \in (0, 1)$, for sufficiently small $\delta$ and sufficiently large $\lambda_1$ depending only on $n$, $k$, $\gamma$, and $\sigma_k(\lambda)$, the following inequality holds at $\lambda$:
	\begin{equation}
		\label{positive-c}
		- \sum_{p\neq q} \frac{\sigma_k^{p \ov p, q \ov q} \xi_p \xi_{\ov q}}{\sigma_k} + 2 \frac{|\sum_i \sigma_k^{i \ov i} \xi_i|^2}{\sigma_k^2} 
		+ \sum_{i > 1} \frac{\sigma_k^{i \ov i} |\xi_i|^2}{(1 + \delta )\lambda_1 \sigma_k} 
		\geq (1 - \gamma) \frac{\sigma_k^{1 \ov 1} |\xi_1|^2}{\lambda_1 \sigma_k}
	\end{equation}
for any vector $\xi = (\xi_1, \ldots, \xi_n) \in \mathbb{C}^n$.
\end{lemma}

\begin{proof}
As the inequality is homogeneous of degree $-2$, it is equivalent to verify that 
	\[
		- \sum_{p\neq q} \frac{\s_k^{p \ov p,q \ov q}(\tilde{\lambda})\xi_p \xi_{\ov q}}{\s_k(\tilde{\lambda})}
		+ 2 \frac{|\sum_i \s_k^{i\ov i}(\tilde{\lambda}) \xi_i|^2}{\s_k(\tilde{\lambda})^2} 
		+ \sum_{i > 1} \frac{\s_k^{i \ov i}(\tilde{\lambda}) |\xi_i|^2}{(1 + \delta )\s_k(\tilde{\lambda})} 
		\geq (1 - \gamma) \frac{\s_k^{1\ov 1}(\tilde{\lambda}) |\xi_1|^2}{\s_k(\tilde{\lambda})},
	\]
where $\tilde{\lambda}_i=\frac{\la_i}{\la_1}$ and $\s_k(\tilde{\lambda})$ is the $k$-th elementary symmetric function of $\tilde{\lambda}$. 
Note that
\[
\tilde{\lambda}_1=1, \; \tilde{\lambda}_n\geq -\de \; \mbox{and}\; 
\s_k(\tilde{\lambda})=\frac{\s_k(\lambda)}{\la_1^k}.
\]
Define 
\[
q_k(\tilde{\lambda}) = \frac{\s_k(\tilde{\lambda})}{\s_{k-1}(\tilde{\lambda})}\; \mbox{and}\;
q_{k;i}(\tilde{\lambda}) = \frac{\s_{k;i}(\tilde{\lambda})}{\s_{k-1;i}(\tilde{\lambda})}.
\]
We first have the following
\begin{equation}
\label{main}
\begin{aligned}
&\; - \sum_{p\neq q} \frac{\s_k^{p \ov p,q \ov q} \xi_p \xi_{\ov q}}{\s_k}
		+ 2 \frac{|\sum_i \s_k^{i\ov i} \xi_i|^2}{\s_k^2}\\
        = &\; - \partial_{\ov \xi} \partial_\xi( \log q_k+\log \s_{k-1})+ |\partial_\xi \log\s_k|^2\\
	= &\; - \frac{\partial_{\ov \xi} \partial_\xi q_k}{q_k} + |\partial_\xi \log \s_k-\partial_\xi\log\s_{k-1}|^2
	- \partial_{\ov \xi} \partial_\xi \log \s_{k-1} + |\partial_\xi \log\s_k|^2\\
	\geq &\; - \frac{\partial_{\ov \xi} \partial_\xi q_k}{q_k} + \frac{1}{2} |\partial_\xi \log \s_{k-1}|^2 - \partial_{\ov \xi} \partial_\xi \log \s_{k-1},
\end{aligned}
\end{equation}
where we used the Cauchy-Schwarz inequality in the last inequality. Define
\[
 \hat \xi = (0, \xi_2, \cdots, \xi_n).
\]
Then $\partial_\xi = \xi_1 \partial_1 + \partial_{\hat \xi}$.
For the second term of \eqref{main},
\begin{equation}\label{f1}
|\partial_\xi \log \s_{k-1}|^2 = |\frac{\s_{k-2;1} \xi_1}{\s_{k-1}} + \partial_{\hat \xi} \log \s_{k-1}|^2.    
\end{equation}
For the last term of \eqref{main}
\begin{equation}\label{f2}
\begin{aligned}
 \partial_{\ov \xi} \partial_\xi \log \s_{k-1} = &\; \partial_{\ov {\hat \xi}}\partial_{\hat \xi} \log \s_{k-1} 
- \Big| \frac{\s_{k-2; 1} \xi_1}{\s_{k-1}} \Big|^2 \\
&\; + \sum_{j>1} \frac{\s_{k-1} \s_{k-3;1j} - \s_{k-2;1} \s_{k-2;j}}{\s_{k-1}^2} \xi_{\ov 1} \xi_j\\
&\; + \sum_{j>1} \frac{\s_{k-1} \s_{k-3;1j} - \s_{k-2;1} \s_{k-2;j}}{\s_{k-1}^2} \xi_1 \xi_{\ov j}.   
\end{aligned}
\end{equation}
And we note that \begin{equation}\label{concavity}
	- \partial_{\ov {\hat \xi} }\partial_{\hat \xi} \log \s_{k-1} = - \sum_{i=1}^{k-1} \frac{\partial_{\ov {\hat \xi} } \partial_{\hat \xi}q_i}{q_i} + \sum_{i=1}^{k-1} |\partial_{\hat \xi} \log q_i|^2 \geq \sum_{i=1}^{k-1} |\partial_{\hat \xi}\log q_i|^2.
	\end{equation}
By \eqref{f1}, \eqref{f2} and \eqref{concavity},
we estimate the last two terms of \eqref{main} as follows.
\begin{equation}\label{main-0}
 \begin{aligned}
&\; \frac{1}{2} |\partial_\xi \log \s_{k-1}|^2 - \partial_{\ov \xi} \partial_\xi \log \s_{k-1}\\
\geq &\; \frac{1}{k} \Big|\frac{\s_{k-2;1} \xi_1}{\s_{k-1}} \Big|^2 + \Big| \frac{\s_{k-2; 1} \xi_1}{\s_{k-1}} \Big|^2 - 2\sum_{j>1} \mbox{Re}\left(\frac{\s_{k-1} \s_{k-3;1j} - \s_{k-2;1} \s_{k-2;j}}{\s_{k-1}^2} \xi_{\ov 1} \xi_j \right),
\end{aligned}   
\end{equation}
where we use the following Cauchy-Schwarz inequality
\[\begin{aligned}
\Big|\frac{\s_{k-2;1} \xi_1}{\s_{k-1}} + \sum_{i=1}^{k-1} \partial_{\hat \xi} \log q_i\Big|^2
\geq \frac{1}{k} \Big|\frac{\s_{k-2;1} \xi_1}{\s_{k-1}} \Big|^2
 - \sum_{i=1}^{k-1} |\partial_{\hat \xi} \log q_i|^2.
\end{aligned}\]
Let us define the following notation
\[
I_{1j} : = \frac{\s_{k-1} \s_{k-3;1j} - \s_{k-2;1} \s_{k-2;j}}{\s_{k-1}^2}. 
\]
Plugging \eqref{main-0} into \eqref{main}, we obtain that
\begin{equation}
	\label{main-1}
	\begin{aligned}
		&\; - \sum_{p\neq q} \frac{\s_k^{p\ov p,q\ov q} \xi_p \xi_{\ov q}}{\s_k}
		+ 2 \frac{|\sum_i \s_k^{i\ov i} \xi_i|^2}{\s_k^2}\\
		\geq &\; - \frac{\partial_{\ov \xi} \partial_\xi q_k}{q_k} + \Big( 1 + \frac{1}{2k} \Big) \Big|\frac{\s_{k-2;1} \xi_1}{\s_{k-1}}\Big|^2 - 2 \sum_{j>1} \mbox{Re} (I_{1j} \xi_{\ov 1} \xi_j ).
	\end{aligned}
\end{equation}

Referring to the formulas (3.7)-(3.10) in \cite{DXZ}, also (3.8)-(3.14) in \cite{Zhang}), we can estimate the second term in \eqref{main-1} as  
\begin{equation}
	\label{s-k-2-1}
\frac{\s_{k-2;1}}{\s_{k-1}} = \frac{\s_{k-1} - \s_{k-1;1}}{\tilde{\lambda}_1 \s_{k-1}} \geq \frac{1 - C \delta}{\tilde{\lambda}_1};
\end{equation}
for the last term in \eqref{main-1}, we can estimate the last term in \eqref{main-1} as 
\begin{align}\label{xk'}
- 2 \mbox{Re} (I_{1j} \xi_{\ov 1} \xi_j )
\geq -2 \frac{C \tilde{\lambda}_k^2}{\tilde{\lambda}_j^2} |\xi_{\ov 1} \xi_j|
\geq - \varepsilon |\xi_1|^2 - \frac{C\tilde{\lambda}_k^4 |\xi_j|^2}{\varepsilon  \tilde{\lambda}_j^4}
\end{align}
for $1 < j \leq k-1$; 
\begin{align}\label{xk''}
  - 2 \mbox{Re} (I_{1j} \xi_{\ov 1} \xi_j )
\geq -2 C |\xi_{\ov 1} \xi_j|
\geq - \varepsilon |\xi_1|^2 - \frac{C |\xi_j|^2}{\varepsilon}
\end{align}
for $k \leq j \leq n$, where $\varepsilon > 0$ is arbitrary and $C$ depends only on $n$ and $k$. By the same arguments in \cite{DXZ,Zhang}, we note that, if $\s_{k;j}>0$ and 
$\tilde{\lambda}_j\geq (C\de)^{\frac{1}{k-1}}$, then $1\leq j\leq k-1$, and
\begin{equation}
\label{I-j-1}
\frac{C}{\tilde{\lambda}_j}\leq \frac{\s_{k-1;j}}{\s_k};
\end{equation}
if $\s_{k;j}>0$ and $\tilde{\lambda}_j< (C\de)^{\frac{1}{k-1}}$,
\begin{equation}
\label{I-j-2}
C\leq \frac{\de^{\frac{1}{k-1}}\s_{k-1;j}}{\s_k};
\end{equation}
and, if $\s_{k;j}\leq 0$, 
\begin{equation}
\label{I-j-3}
 \frac{1}{\tilde{\lambda}_j}\leq \frac{\s_{k-1;j}}{\s_k}.
\end{equation}
By Lemma \ref{xk} and the assumption that $|\tilde{\lambda}_n|<\delta$, we obtain 
\begin{align}
    \tilde{\lambda}_k \leq C \delta\leq C\delta^{\frac{1}{k-1}},
\end{align}
where we assume that $\lambda_1$ is sufficiently large depending on $\de$.
Then, inserting \eqref{I-j-1}, \eqref{I-j-2}, \eqref{I-j-3} into \eqref{xk'} and \eqref{xk''}, we deduce that
\begin{align}\label{xk'''}
 - 2 \sum_{j>1} \mbox{Re} (I_{1j} \xi_{\ov 1} \xi_j ) \geq - n \varepsilon |\xi_1|^2 - \frac{C\delta^{\frac{1}{k-1}}}{\varepsilon} 
\sum_{i > 1} \frac{\s_k^{i\ov i} |\xi_i|^2}{ \s_k}.    
\end{align}
Let $\tilde{\de} = \delta^{\frac{1}{k-1}}$.
Plugging \eqref{s-k-2-1} and \eqref{xk'''} into \eqref{main-1}, 
we arrive at the conclusion that 
\begin{equation}
	\label{main-2}
	\begin{aligned}
	&\; - \sum_{p\neq q} \frac{\s_k^{p\ov p, q\ov q} \xi_p \xi_{\ov q}}{\s_k}
	+ 2 \frac{|\sum_i \s_k^{i\ov i} \xi_i|^2}{\s_k^2} 
	+ \sum_{i > 1} \frac{\s_k^{i\ov i} |\xi_i|^2}{(1 + \delta ) \s_k} \\
	\geq &\; - \frac{\partial_{\ov \xi} \partial_\xi q_k}{q_k} + \Big( 1 + \frac{1}{2k} 
	- 2n \varepsilon \Big) \Big(1-C\delta\Big)^2 |\xi_1|^2
	+ \Big( 1 - \frac{ C\tilde \delta }{\varepsilon} \Big)
	\sum_{i > 1} \frac{\s_k^{i\ov i} |\xi_i|^2}{\s_k}.
\end{aligned}
\end{equation}
We denote by $Q$ the following quadratic form:
\[
Q:=  - \sum_{p\neq q} \frac{\s_k^{p\ov p, q\ov q} \xi_p \xi_{\ov q}}{\s_k}
+ 2 \frac{|\sum_i \s_k^{i\ov i} \xi_i|^2}{\s_k^2} 
+ \sum_{i > 1} \frac{\s_k^{i\ov i} |\xi_i|^2}{(1 + \delta ) \s_k}.
\]
We further derive from \eqref{main-2} that, for $\varepsilon = 1/8nk$, there exists a constant $C$ depending only on $n$, $k$ and $\s_k$, such that
\begin{equation}
	\label{main-3}
	\begin{aligned}
		Q \geq & - \frac{\partial_{\ov \xi} \partial_\xi q_k}{q_k} + \Big( 1 + \frac{1}{4k} \Big) \Big(1-C\delta\Big)^2 |\xi_1|^2
		+ \Big( 1 - C\tilde \delta \Big)
		\sum_{i > 1} \frac{\s_k^{i\ov i} |\xi_i|^2}{ \s_k}.
	\end{aligned}
\end{equation}
For sufficiently small $\delta$,
we have 
\begin{equation}
	\label{main-4}
	\begin{aligned}
		Q \geq - \frac{\partial_{\ov \xi} \partial_\xi q_k}{q_k} + \Big( 1 + \frac{1}{8k} \Big) |\xi_1|^2
		+ \Big( 1 - C\tilde \delta \Big)
		\sum_{i > 1} \frac{\s_k^{i\ov i} |\xi_i|^2}{\s_k}.
	\end{aligned}
\end{equation}
We now assume that
\[
\s_{k;1} \geq - \frac{1}{16k} \s_k.
\]
Then, from $\s_k = \tilde{\lambda}_1 \s_{k-1;1} + \s_{k;1}$, it follows that 
\[
1 + \frac{1}{16k}\geq\frac{\s_{k-1;1}}{\s_k}.
\]
Therefore, we obtain for sufficiently small $\tilde{\delta}$ that
\begin{equation}
	\label{main-5}
	\begin{aligned}
		Q \geq \Big( 1 + \frac{1}{1+ 16k} \Big) \frac{\s_{k}^{1\ov 1} |\xi_1|^2}{ \s_k} \geq \frac{\s_{k}^{1\ov 1} |\xi_1|^2}{ \s_k},
	\end{aligned}
\end{equation}
which proves the inequality stated in \eqref{positive-c} with any $\gamma\geq 0$.

Next, we consider the case where
\begin{equation}
	\label{sit-2}
	\s_{k;1} < - \frac{1}{16k} \s_k.
\end{equation}
To proceed, referring to the Claims in Lemma 3.1 of \cite{DXZ}, also Lemma 1.1 in \cite{Zhang}, we claim the following three facts under the assumption \eqref{sit-2}.

\textit{\textbf{Claim 1}:} We have
\begin{equation}
	\label{z-2'}
	\tilde{\lambda}_k \leq C |\tilde{\lambda}_n|
\end{equation}
for some constant $C$ depending on $n$ and $k$. 

\textit{\textbf{Claim 2}:} There exist two positive constants $c_1$ and $c_2$ only depending on $n$, $k$ and $\s_k$, such that
\begin{equation}
	\label{claim-2}
	-\s_{k+1} \geq c_1 \tilde{\lambda}_1 \cdots \tilde{\lambda}_{k-1} \tilde{\lambda}_k^2 \geq c_2\s_k.
\end{equation}

\textit{\textbf{Claim 3}:} There exist positive constants $c_3$, $c_4$ and $c_5$ only depending on $n$, $k$ and $\s_k$, such that
\begin{equation}
\label{claim-3}
c_3 \tilde{\lambda}_2 \cdots \tilde{\lambda}_k^2 \leq \s_{k-1;1} \leq c_4 \tilde{\lambda}_2 \cdots \tilde{\lambda}_k^2\leq c_5\de^2\s_{k-1}.
\end{equation}

We now continue with the proof.
Let $e_1 = (1, 0, \ldots, 0)$. 
We now decompose $\xi$ as
\begin{equation}
	\label{decomp}
\xi =(1+a)\xi_1 \tilde{\lambda}- a \xi_1 e_1 + \zeta.
\end{equation}
for some $a \in \mathbb{C}$.
Here $\zeta = (0, \zeta_2, \ldots, \zeta_n) \in \mathbb{C}^n$, which satisfies
\begin{equation}
	\label{perp}
(\zeta_k, \ldots, \zeta_n) \perp (\lambda_k, \ldots, \lambda_n).
\end{equation}
Set
\[
II_{1p}' = \frac{\s_{k-2;1p}}{\s_k} - 
\frac{\s_{k-1;p}}{\s_k} \frac{\s_{k-2;1}}{\s_{k-1}} 
- \frac{\s_{k-1;1}}{\s_k} \frac{\s_{k-2;p}}{\s_{k-1}}
+ \frac{\s_{k-2;1}}{\s_{k-1}} \frac{\s_{k-2;p}}{\s_{k-1}}
\]
and
\[
II_{1p}'' = \frac{\s_{k-2;1}}{\s_{k-1}} \frac{\s_{k-2;p}}{\s_{k-1}}
- \frac{\s_{k-3;1p}}{\s_{k-1}}.
\]
As was shown in \cite{DXZ, Zhang}, we have 
\begin{equation}
	\label{q''}
	\begin{aligned}
		-\frac{\partial_{\ov \xi} \partial_\xi q_k}{q_k} 
		= &\; - (-a\xi_1 e_1 + \zeta)_p \frac{q_k^{p\ov p, q\ov q}}{q_k} (\ov{- a\xi_1 e_1 + \zeta})_q\\
		= &\; - 2|a|^2 |\xi_1|^2 \frac{\s_{k-2;1}}{\s_{k-1}} 
		\Big(\frac{\s_{k-2;1}}{\s_{k-1}} - \frac{\s_{k-1;1}}{\s_k}\Big) \\
		&\; - \frac{\partial_{\ov \zeta} \partial_\zeta q_k}{q_k} + \ov a \sum_{p>1} \xi_{\ov1} \zeta_p II_{1p} + a \sum_{p>1} \xi_1 \zeta_{\ov p} II_{1p},
	\end{aligned}
\end{equation}
where
\[\begin{aligned}
II_{1p} 
= &\; II_{1p}' + II_{1p}''.
\end{aligned}\]
Also, by the arguments in \cite{DXZ} and \cite{Zhang} we obtain that,
for the first term in \eqref{q''},
\begin{equation}\label{1st}
    \begin{aligned}
- 2 |a|^2 |\xi_1|^2 \frac{\s_{k-2;1}}{\s_{k-1}} 
\Big(\frac{\s_{k-2;1}}{\s_{k-1}} - \frac{\s_{k-1;1}}{\s_k}\Big) 
\geq - 2 |a|^2 |\xi_1|^2\Big(1 - C\delta \Big)
\frac{ \s_{k;1}}{ \s_k};
\end{aligned}
\end{equation}
for the last two terms in \eqref{q''}, if $2\leq j\leq k-1$, then
\begin{equation}
	\label{2a}
	\begin{aligned}
		 2 \mbox{Re} (\ov a \xi_{\ov1} \zeta_i II_{1i})
		\geq &\; - 2C |\ov a \xi_{\ov1} \zeta_i|\frac{ \s_{k-1;1} }{\s_k}
	\frac{\tilde{\lambda}_k}{\tilde{\lambda}_i^2} \\
		\geq &\; - \varepsilon \frac{\s_{k-1;1} }{\s_k\tilde{\lambda}_i^3} |\zeta_i|^2 
		- \frac{C |a|^2 \tilde{\lambda}_k^2}{\varepsilon} \frac{\s_{k-1;1} |\xi_1|^2}{\s_k\tilde{\lambda}_i };
	\end{aligned}
\end{equation}
if $k \leq p \leq n$, then
\begin{equation}
	\label{2a'}
	\begin{aligned}
		 2 \mbox{Re} (\ov a \xi_{\ov 1} \zeta_p II_{1p})
		\geq &\; - 2C |\ov a \xi_{\ov 1} \zeta_p| \frac{\s_{k-1;1}}{\s_k} \frac{1}{\tilde{\lambda}_{k-1}} \\
		\geq &\; - \varepsilon \frac{\s_{k-1;1}  }{\s_k\tilde{\lambda}_{k-1} \tilde{\lambda}_k^2} |\zeta_p|^2
		- \frac{C |a|^2  \tilde{\lambda}_k^2}{\varepsilon} \frac{\s_{k-1;1} |\xi_1|^2}{\s_k\tilde{\lambda}_{k-1} }.
		\end{aligned}
\end{equation}
Substituting \eqref{1st}, \eqref{2a}, and \eqref{2a'} into \eqref{q''}, we obtain the following.
\begin{equation}
	\label{q''-1}
	\begin{aligned}
		-\frac{\partial_{\ov \xi} \partial_\xi q_k}{q_k} 
		\geq &\; - \frac{\partial_{\ov \zeta} \partial_\zeta q_k}{q_k} - 2 |a|^2 |\xi_1|^2 \Big(1 - C\delta \Big)
		\frac{ \s_{k;1}}{\s_k}\\
        &\; - \varepsilon \sum_{2 \leq i < k} \frac{\s_{k-1;1} }{\s_k\tilde{\lambda}_i^3} |\zeta_i|^2 - \varepsilon \sum_{k \leq p \leq n} \frac{\s_{k-1;1}  }{\s_k\tilde{\lambda}_{k-1} \tilde{\lambda}_k^2} |\zeta_p|^2\\
		&\; - \frac{C |a|^2 \tilde{\lambda}_k^2 }{\varepsilon} \frac{\s_{k-1;1} |\xi_1|^2 }{\s_k \tilde{\lambda}_{k-1}}  
		- \frac{C |a|^2 \tilde{\lambda}_k}{\varepsilon} \frac{\s_{k-1;1} |\xi_1|^2}{ \s_k}.
	\end{aligned}
\end{equation}
By Lemma \ref{qk-1}, there exists a positive constant $b$ depending on $n$ and $k$ such that 
\[\begin{aligned}
-\frac{\partial_{\ov \zeta} \partial_\zeta q_k}{q_k} 
\geq &\; b \sum_{2 \leq i < k} \frac{\s_{k-1;1}}{\s_k\tilde{\lambda}_i^3} |\zeta_i|^2 
+ b \sum_{k \leq p \leq n} \frac{\s_{k-1;1}}{\s_k\tilde{\lambda}_{k-1} \tilde{\lambda}_k^2} |\zeta_p|^2,
\end{aligned}\] 
where we used Claim 3 in the last inequality.
Let $ \varepsilon = b/2$. Then, it follows from \eqref{q''-1} that
\begin{equation}
	\label{q''-2}
	\begin{aligned}
		-\frac{\partial_{\ov \xi} \partial_\xi q_k}{q_k} 
		\geq &\; \frac{b}{2} \sum_{2 \leq i < k} \frac{\s_{k-1;1} }{\s_k\tilde{\lambda}_i^3} |\zeta_i|^2 + \frac{b}{2} \sum_{k \leq p \leq n} \frac{\s_{k-1;1} }{\s_k\tilde{\lambda}_{k-1} \tilde{\lambda}_k^2} |\zeta_p|^2 \\
		&\; -2 |a|^2 |\xi_1|^2 \Big(1 - C\delta \Big)
		\frac{\s_{k;1}}{ \s_k} - \frac{C |a|^2 \de}{b} \frac{\s_{k-1;1} |\xi_1|^2}{\s_k}.
	\end{aligned}
\end{equation}
Inserting \eqref{q''-2} into \eqref{main-3},
we arrive at 
\begin{equation}
	\label{main-6}
	\begin{aligned}
		Q \geq &\;\frac{b}{2} \sum_{2 \leq p < k} \frac{\s_{k-1;1} }{\s_k\tilde{\lambda}_p^3} |\zeta_p|^2 + \frac{b}{2} \sum_{k \leq p \leq n} \frac{\s_{k-1;1} }{\s_k\tilde{\lambda}_{k-1} \tilde{\lambda}_k^2} |\zeta_p|^2\\
        &\; - 2 |a|^2 \Big(1 - C\delta \Big) \frac{ \s_{k;1}}{ \s_k} |\xi_1|^2 + \Big( 1 + \frac{1}{4k} \Big) \Big(1 - C\delta\Big) |\xi_1|^2\\
		&\; - \frac{C |a|^2 \de}{b} \frac{\s_{k-1;1} }{\s_k} |\xi_1|^2
	  + \Big( 1 - C \tilde \delta \Big)
		\sum_{i > 1} \frac{\s_k^{i\ov i} |\xi_i|^2}{ \s_k}.
	\end{aligned}
\end{equation}

We now address the last term in \eqref{main-6}. \\
\textbf{Notation:} there exists some $1\leq \ell \leq k-1$, such that $\tilde{\lambda}_\ell \geq\sqrt{\de}$ and $\tilde{\lambda}_{\ell+1} < \sqrt{\de} $. 
Using \eqref{decomp} we have 
 \begin{equation}
 	\label{i>1}
 \begin{aligned}
 	\sum_{i > \ell} \frac{\s_k^{i\ov i} |\xi_i|^2}{ \s_k}
 	\geq &\; |1+a|^2 |\xi_1|^2 \sum_{i > \ell} \frac{\s_k^{i\ov i} \tilde{\lambda}_i^2}{\s_k} - 2\Big|(1+a) \xi_1 \sum_{i > \ell} \frac{\s_k^{i\ov i} \tilde{\lambda}_i \zeta_{\ov i}}{ \s_k}\Big|\\
 	\geq &\; |1+a|^2 |\xi_1|^2 \sum_{i > \ell} \frac{\s_k^{i\ov i} \tilde{\lambda}_i^2}{ \s_k} - 2\Big|(1+a) \xi_1 \sum_{k \leq p\leq n} \frac{\s_{k-2;p} \tilde{\lambda}_p^2 }{ \s_k} \zeta_{\ov p} \Big|\\
 	&\; - 2\Big|(1+a) \xi_1 \sum_{\ell < i\leq k-1} \frac{\s_{k-1;i} \tilde{\lambda}_i}{ \s_k} \zeta_{\ov i} \Big|.
 \end{aligned}
 \end{equation}
Here by \eqref{perp}, we note that
 \[
 \sum_{k\leq p \leq n} \s_{k-1;p} \tilde{\lambda}_p \zeta_{\ov p} =
 \sum_{k\leq p \leq n} (\s_{k-1} \tilde{\lambda}_p \zeta_{\ov p} - \s_{k-2;p} \tilde{\lambda}_p^2 \zeta_{\ov p})
 = - \sum_{k\leq p \leq n} \s_{k-2;p} \tilde{\lambda}_p^2 \zeta_{\ov p}.
 \]
By Lemma \ref{sigmak} (3), we obtain that
\begin{equation}
	\label{z-1'}
\begin{aligned}
	\sum_{\ell + 1 \leq j \leq n} \s_{k-1;j} \tilde{\lambda}_j^2 
	=& \sum_{\ell + 1 \leq i \leq n} \tilde{\lambda}_i \s_k - \sum_{1 \leq i \leq \ell} \s_{k+1;i}-(k+1-\ell)\s_{k+1}\\
    \geq &\; - \frac{k+1-\ell}{1 + C \delta} \frac{ \s_{k;1} }{\s_k}
	- C\sqrt{\de} \frac{\s_{k-1;1}}{ \s_k},
\end{aligned}
\end{equation}
where we use Lemma \ref{q2} (5), \eqref{z-1'} and Claim 3 in the last inequality. As was shown in \cite{DXZ, Zhang}, we observe that, for $k \leq p \leq n$, 
\begin{equation}
	\label{p>=k}
\begin{aligned}
	2\Big|(1+a) \xi_1 \frac{\s_{k-2;p} \tilde{\lambda}_p^2}{ \s_k} \zeta_{\ov p} \Big|
	\leq &\; \frac{b}{4} \frac{\s_{k-1;1} }{\s_k\tilde{\lambda}_{k-1} \tilde{\lambda}_k^2} |\zeta_p|^2 + \frac{C\de |1+a|^2}{b}\frac{ \s_{k-1;1} }{\s_k} |\xi_1|^2;
\end{aligned}
\end{equation}
for $\ell < i \leq k-1$, 
\begin{equation}
	\label{p<k-1}
\begin{aligned}
	 2\Big|(1+a) \xi_1 \frac{\s_{k-1;i} \tilde{\lambda}_i}{ \s_k} \zeta_{\ov i} \Big|
	  \leq &\; \frac{b}{4} \frac{\s_{k-1;1} |\zeta_i|^2}{\s_k\tilde{\lambda}_i^3}
	  + \frac{C \sqrt{\de} |1+a|^2 }{b}\frac{\s_{k-1;1}}{\s_k} |\xi_1|^2.	
\end{aligned}
\end{equation}
Substituting the inequalities \eqref{z-1'}, \eqref{p>=k} and \eqref{p<k-1} into \eqref{i>1}, we conclude that
\[
\begin{aligned}
	\sum_{i > \ell} \frac{\s_k^{i\ov i} |\xi_i|^2}{\s_k}
	\geq &\; - |1+a|^2 \Big(\frac{k+1 -\ell}{1 + C\delta} \frac{ \s_{k;1} }{ \s_k} |\xi_1|^2 
	+ C\sqrt{\de} \frac{\s_{k-1;1}}{\s_k} |\xi_1|^2 \Big)\\	
    &\; - \frac{b}{4} \sum_{k \leq p\leq n} \frac{\s_{k-1;1}}{\s_k\tilde{\lambda}_{k-1} \tilde{\lambda}_k^2} |\zeta_p|^2 -\frac{C \de |1+a|^2 }{b}\frac{\s_{k-1;1}}{\s_k} |\xi_1|^2\\
	&\; - \frac{b}{4} \sum_{\ell < i \leq k-1} \frac{\s_{k-1;1} }{\s_k\tilde{\lambda}_i^3} |\zeta_i|^2
	- \frac{C \sqrt{\de} |1+a|^2 }{b} \frac{\s_{k-1;1}}{\s_k} |\xi_1|^2.
\end{aligned}
\]
Denote by $\tilde a$ the quantity $|a|^2 + |1 + a|^2$. Then \eqref{main-6} reduces to:
\begin{equation}
	\label{main-7-c}
	\begin{aligned}
		Q \geq & - 2\tilde a  (1 - C \tilde \delta )
		\frac{\s_{k;1}}{\s_k} |\xi_1|^2 + ( 1-C\delta ) \Big(1 + \frac{1}{4k}\Big) |\xi_1|^2 - \frac{C \tilde a \sqrt{\de}}{b} \frac{\s_{k-1;1}}{\s_k} |\xi_1|^2\\
    \geq &  (1 - C \tilde{\delta} )
		\frac{\s_{k-1;1}}{\s_k} |\xi_1|^2 +\tilde a \left(\de^{\frac{1}{n}}- \frac{C \sqrt{\de}}{b} \right)\frac{\s_{k-1;1}}{\s_k} |\xi_1|^2,    
	\end{aligned}
\end{equation}
where we use $k + 1 - \ell \geq 2$ in the first inequality; $\tilde a \geq 1/2$ and $\s_{k-1;1} = \s_k - \s_{k;1}$ in the second inequality. Choose sufficiently small $\delta$ such that the last term of \eqref{main-7-c} is negative and $C\tilde{\de}<\gamma$.
At last, we obtain
\begin{equation}
	\label{main-10-c}
		Q \geq (1 - \gamma) \frac{\s_{k-1;1} }{ \s_k} |\xi_1|^2,
\end{equation}
which completes the proof of Lemma \ref{key-c}.
\end{proof}

\section{Proof of Theorem \ref{thm1}}

Now we begin the proof of Theorem \ref{thm1}.
We apply the maximum principle to the following test function:
\begin{equation}
	\label{test}
	G = \log \lambda_{max} + \varphi(|Du|^{2}) + \phi(u),
\end{equation}
where $\la_{max} = \max_{1\leq i \leq n}\{\la_i\}$,
\[
\varphi (t) = e^{Nt} \;\mbox{and}\;  \phi(s) = e^{ K (-s + ||u||_{C^1}+1)}.
\]
When $K \gg N \gg 1$, 
the function $\varphi$ and $\phi$ satisfy the following properties:
\begin{equation}
	\label{testfun}
	\varphi''- 2\phi'' \Big( \frac{\varphi'}{\phi'} \Big)^{2} > 0, \;\varphi' > 0, \; \phi' < 0, \; \phi''>0.
\end{equation}
We may assume that the maximum of $G$ is achieved at some point $p \in M$.
We choose a local coordinate system centered at a point $p \in M$ such that the Hermitian metric $\omega$ satisfies 
\[
\omega = \sqrt{-1} \delta_{k \ell} dz^{k} \wedge d \overline{z}^{\ell}
\]
at $p$, 
and the matrix $g_{i \ov{j}}$ of the real $(1,1)$-form $g$ is diagonal at $p$. Let $\mu$ denote the multiplicity of the largest eigenvalue of $g$.
We denote its eigenvalues by
\[
\lambda_{1} = \lambda_{2} = \cdots = \lambda_\mu > \lambda_{\mu+1} \geq \cdots \geq \lambda_{n}
\]
so that $g_{i \overline{j}} = \lambda_i \delta_{ij}$ at the point $p$.
To overcome $\la_1$ being not differentiable, we define a smooth function
$f$ on $M$ by the following equation
\[
G (p) \equiv \log f + \varphi(|Du|^{2}) + \phi(u) = : \hat G.
\]
Note that 
\[
f \geq \la_1 \;\mbox{on}\; M\; \mbox{and}\; f = \la_1 \; \mbox{at} \; p. 
\]
By Lemma 3.2 in \cite{T-W19}, for each $i$, the following formulas
\begin{equation}
    \label{f-1}
    g_{k\ov l i} = f_i g_{k\ov l}, \; \forall \; 1 \leq k,l \leq \mu,
\end{equation}
and
\begin{equation}
    \label{f-2}
    f_{i \ov i} \geq g_{1\ov 1 i \ov i} + \sum_{\ell > \mu} 
    \frac{|g_{\ell \ov 1 i}|^2 + |g_{1 \ov \ell i}|^2}{\la_1 -\la_{\ell}}
\end{equation}
holds at $p$.
Differentiate the function $\hat G$ at $p$. We get
\begin{equation}
	\label{Diff1}
	\frac{g_{1\ov 1 i}}{\lambda_1} + \varphi' \sum_l (u_l u_{\ov l i} + u_{\ov l} u_{li}) + \phi' u_i = 0,
\end{equation} 
and 
\begin{equation}
	\label{Diff2}
	0 \geq \frac{g_{1\ov 1 i \ov i}}{\lambda_1} 
	+ \sum_{\ell > \mu} \frac{|g_{\ell \ov 1 i}|^2 + |g_{1 \ov \ell i}|^2}{(\lambda_1 - \lambda_\ell) \lambda_1} - \frac{|g_{1 \ov 1 i}|^2}{\lambda_1^2} + \varphi_{i \ov i} + \phi_{i \ov i}.
\end{equation}
Contracting the above inequality with $F^{i \ov i}$, and by \eqref{dif-eqn2}, we obtain
\begin{equation}
	\label{G1}
	\begin{aligned}
		0 \geq &\;  \frac{1}{\lambda_1} \Big\{ 2 \text{Re} \sum_l ( \psi_{v_l} g_{1 \ov{1} l} )
		- F^{p \overline{q}, r \overline{s}}  g_{p \overline{q} 1} g_{r \overline{s} \ov{1}}
		 - 2\text{Re} F^{i \ov{i}} T_{i1}^{t} u_{t \overline{i} \overline{1}} \Big\}\\
		&\; + \frac{1}{\lambda_1} F^{i \overline{i}}(a_1 u_{\ov 1 i \ov i} + a_{\ov 1} u_{1 i \ov i} - a_i u_{\ov i 1 \ov 1} - a_{\ov i} u_{i 1 \ov 1}) - C \mathcal{F} - C\\
		&\; - \frac{1}{4\lambda_1} ( |D \ov{D} u|_{F \omega}^{2} + |DDu|_{F \omega}^{2} ) 
		- \frac{C}{\lambda_1} ( |DDu|^{2} + |D \ov{D}u|^{2} ) \\
		&\; + F^{i \ov i} \sum_{\ell > \mu} \frac{|g_{\ell \ov 1 i}|^2 + |g_{1 \ov \ell i}|^2}{(\lambda_1 - \lambda_\ell) \lambda_1} - \frac{F^{i \ov i} |g_{1 \ov 1 i}|^2}{\lambda_1^2} + F^{i \ov i} \varphi_{i \ov i} + F^{i \ov i} \phi_{i \ov i}.
	\end{aligned}
\end{equation} 
By \eqref{Diff1}, we see that
\[
\frac{\psi_{v_i} g_{1\ov 1 i}}{\lambda_1} + \varphi' \sum_m \psi_{v_i} (u_m u_{\ov m i} + u_{\ov m} u_{mi}) + \phi' \psi_{v_i} u_i = 0.
\]
Therefore, we obtain
\[
2 \text{Re} \frac{\psi_{v_l} g_{1\ov 1 l}}{\lambda_1}  
+ 2 \varphi' \text{Re} \sum_{m} \psi_{v_l} (u_m u_{\ov m l} + u_{\ov m} u_{ml}) \geq C \phi', 
\]
where $C$ depends on $||u||_{C^1}$ and $\psi$.
By \eqref{derivative of Du'}, \eqref{D-u} and the above inequality, we have
\[\begin{aligned}
&\; 2 \text{Re} \frac{\psi_{v_l} g_{1\ov 1 l}}{\lambda_1} 
+ F^{i \overline{i}} \varphi_{i \overline{i}}
+ F^{i \ov i} \phi_{i \ov i}\\
\geq &\; \varphi'' F^{i \ov i} |(|Du|^2)_i|^2 + \phi'' F^{i \ov i} |u_i|^2
 + \frac{\varphi'}{2} ( |D D u|_{ F\omega}^{2} + |D \ov{D} u|_{F \omega}^{2} ) \\
 &\; - \varepsilon \phi' \mathcal{F} - \phi' F^{i\bar i} ( a_i u_{\bar i} + a_{\bar i} u_i ) - C \varphi'\mathcal{F} - C\varphi' + C\phi'.
\end{aligned}\]
By the critical equation \eqref{Diff1}, we also have
\[
		\phi'^2 |u_i|^2
		\geq \frac{1}{2} \frac{| g_{1\ov 1 i}|^2}{\lambda_1^2 } - \varphi'^2 | (|Du|^2)_i |^2
\]
and
\[
	\begin{aligned}
		- 2 \phi' F^{i \ov i} \operatorname{Re} \{a_i u_{\ov i} \}
		= &\ 2 F^{i \ov i } \operatorname{Re} \left\{ a_i \Big( \frac{ g_{1 \ov 1 \ov i} }{\lambda_1} + \varphi' (|Du|^2)_{\ov i} \Big) \right\} \\
		\geq &\ - \frac{\phi''}{4 {\phi'}^2} \frac{ F^{i \ov i} |g_{1 \ov 1 \ov i} |^2 }{\lambda_1^2} -  {\varphi'}^2 \frac{\phi''}{\phi'^2} F^{i \ov i} | (|Du|^2)_i |^2 - C\frac{\phi'^2}{\phi''} \mathcal{F}.
	\end{aligned}
\]
Combining the above two inequalities and
$\varphi''- 2 \phi'' \frac{\varphi'^2}{\phi'^2} \geq 0$,
we arrive at
\begin{equation}
	\label{psi}
\begin{aligned}
	&\; 2 \text{Re} \frac{\psi_{v_l} g_{1\ov 1 l}}{\lambda_1}  
	+ F^{i \overline{i}} \varphi_{i \overline{i}}
	+ F^{i \ov i} \phi_{i \ov i}\\
	\geq &\; \frac{\phi''}{4 {\phi'}^2} \frac{F^{i \ov i} |g_{1 \ov 1 \ov i} |^2 }{\lambda_1^2}
	+ \frac{\varphi'}{2} ( |D D u|_{ F \omega}^{2} + |D \ov{D} u|_{ F \omega}^{2} ) \\
	&\; - \varepsilon \phi' \mathcal{F} - C \Big(\varphi' + \frac{\phi'^2}{\phi''} \Big) \mathcal{F} - C\varphi' + C\phi'.
\end{aligned}
\end{equation}

Now we estimate the torsion term in \eqref{G1}. For any $0 < \alpha < 1$, we have
\begin{equation}
	\label{Torsion}
	\begin{aligned}
		\frac{2 F^{i \ov i}  }{\lambda_1} \text{Re} \Big( \overline{T_{i 1}^{t}} u_{i \ov t 1} \Big)
		\leq &\; \frac{2 F^{i \ov i} }{\lambda_1} |\ov{T_{i1}^t} D_i g_{1\ov t }| + C \sum F^{i \ov i} + \frac{C}{\lambda_1} F^{i \ov i} |u_{1i}|^2\\
		\leq &\;  \frac{\alpha}{2} \frac{F^{i \ov i}}{\lambda_1^2} 
		\sum_t | g_{1\ov t i}|^2 + \frac{C}{\alpha} \sum F^{i \ov i} + \frac{C}{\lambda_1} F^{i \ov i} |u_{1i}|^2.
	\end{aligned}
\end{equation}
Similarly, we can estimate
\begin{equation}
	\label{a-3rd}
	\begin{aligned}
		&\; \frac{1}{\lambda_1} F^{i \ov i}(a_1 u_{\ov 1 i \ov i} + a_{\ov 1} u_{1 i \ov i} - a_i u_{\ov i 1 \ov 1} - a_{\ov i} u_{i 1 \ov 1}) \\
		\leq &\; \frac{\alpha}{4} \frac{ F^{i \ov i} }{\lambda_1^2}
		\Big(|D_i g_{1 \ov i }|^2 + |D_i  g_{1\ov 1 }|^2\Big)
		+ \frac{C}{\alpha} \sum F^{i \overline{i}} +\frac{C}{\lambda_1} F^{i \overline{i}} |u_{1i}|^2\\
		\leq &\; \frac{\alpha}{2} \frac{ F^{i \ov i} }{\lambda_1^2} \sum_{s} | g_{1 \ov s i}|^2 + \frac{C}{\alpha} \sum F^{i \ov i} + \frac{C}{\lambda_1} F^{i \ov i}  |u_{1i}|^2.
	\end{aligned}
\end{equation}

By Lemma \ref{sigmak} (4) and (6), we have 
\[
F^{i \overline{i}} \geq F^{1 \overline{1}} \geq \frac{1}{C \lambda_{1}}
\]
for any fixed $i$. We can estimate that
\begin{equation}
	\label{DDu}
|D D u|_{F \omega}^{2} + |D \ov{D} u|_{F \omega}^{2} \geq \frac{1}{C \lambda_{1}}\left(|D D u|^{2}+|D \overline{D} u|^{2}\right). 
\end{equation}
Substituting \eqref{psi}, \eqref{Torsion}, \eqref{a-3rd} and \eqref{DDu} into \eqref{G1} and by $\lambda_1 \gg 1$, as well as $N \gg 1$, we obtain 
\begin{equation}
	\label{G2}
	\begin{aligned}
		0 \geq &\; - \frac{ F^{p \overline{q}, r \overline{s}}  g_{p \overline{q} 1} g_{r \overline{s} \ov{1}}}{\lambda_1}
		+ \sum_{\ell > \mu} \frac{ F^{i \ov i} (|g_{\ell \ov 1 i}|^2 + |g_{1 \ov \ell i}|^2) }{( \lambda_1 - \lambda_\ell) \lambda_1^2} - \Big(1- \frac{\phi''}{4 {\phi'}^2}\Big) \frac{F^{i \ov i} |g_{1 \ov 1 i}|^2}{\lambda_1^2} \\
		&\; + \frac{\varphi'}{4} ( |DDu|_{F \omega}^{2} + |D \ov{D}u|_{F \omega}^{2} ) - \alpha \sum_{s} \frac{ F^{i \ov i} }{ \lambda_1^2 } | g_{1 \ov s i} |^2 \\
		&\; - \varepsilon \phi' \mathcal{F} - C \Big( \varphi' + \frac{\phi'^2}{\phi''} + \frac{1}{\alpha} \Big) \mathcal{F} - C\varphi' + C\phi'.
	\end{aligned}
\end{equation} 
Note that by \eqref{f-1} we have $g_{1\ov s i} = 0$ for $1 < s \leq \mu$.
We see that
\begin{equation}
	\label{G3}
	\begin{aligned}
		0 \geq &\; - \frac{ F^{p \overline{q}, r \overline{s}}  g_{p \overline{q} 1} g_{r \overline{s} \ov{1}}}{\lambda_1}
		+ \sum_{\ell > \mu} \frac{ F^{i \ov i} |g_{\ell \ov 1 i}|^2 }{ (\lambda_1 - \lambda_\ell) \lambda_1} 
		+ \Big( \frac{1}{1 + \delta} - \alpha \Big) \sum_{\ell > \mu} \frac{F^{i \ov i} |g_{1 \ov \ell i}|^2}{\lambda_1^2}\\
		&\; - \Big(1 - \frac{1}{4 \phi} + \alpha\Big) \frac{F^{i \ov i} |g_{1 \ov 1 i}|^2}{\lambda_1^2} + \frac{\varphi'}{4} ( |DDu|_{F \omega}^{2} + |D \ov{D}u|_{F \omega}^{2} ) \\
		&\; - \varepsilon \phi' \mathcal{F} - C \Big( \varphi' + \phi + \frac{1}{\alpha} \Big) \mathcal{F} - C\varphi' + C\phi'.
	\end{aligned}
\end{equation} 
Take 
\[
\alpha = \frac{1}{16\phi} \; \mbox{and}\; \delta < \frac{1}{16\phi - 1}.
\]
Note that by \eqref{symmetric func 2th deriv},
\[
- F^{p \overline{q}, r \overline{s}} g_{p \overline{q} 1} g_{r \overline{s} \ov 1} = - F^{p \overline{p}, q \overline{q}} g_{p \overline{p} 1} g_{q \overline{q} \ov 1} + F^{p \overline{p}, q \overline{q}}\left| g_{q \overline{p} 1}\right|^{2}.
\]
Adding the term $2 |\psi_1|^2/ (\la_1 F)$ to the right hand side of \eqref{G3} and by the above equation, we therefore obtain 
\begin{equation}
	\label{G4}
	\begin{aligned}
		0 \geq &\; - \frac{ F^{p \overline{p}, q \overline{q}} g_{p \overline{p} 1} g_{q \overline{q} \ov 1} }{\lambda_1} + \frac{ F^{p \overline{p}, q \overline{q}}\left| g_{q \overline{p} 1}\right|^{2}}{\lambda_1} + 2 \frac{|\sum_i F^{ii} g_{i \ov i 1}|^2}{ \lambda_1 F}\\
		&\; + \sum_{\ell > \mu} \frac{ F^{i \ov i} |g_{\ell \ov 1 i}|^2 }{ (\lambda_1 - \lambda_\ell) \lambda_1} 
		+ \Big( 1 - \frac{1}{8\phi} \Big) \sum_{\ell > \mu} \frac{F^{i \ov i} |g_{1 \ov \ell i}|^2}{\lambda_1^2}\\
		&\; - \Big(1 - \frac{3}{16 \phi} \Big) \frac{F^{i \ov i} |g_{1 \ov 1 i}|^2}{\lambda_1^2} + \frac{\varphi'}{4} ( |DDu|_{F\omega}^{2} + |D \ov{D}u|_{F\omega}^{2} ) \\
		&\; - \varepsilon \phi' \mathcal{F} - C ( \varphi' + \phi ) \mathcal{F} - C\varphi' + C\phi' - C \lambda_1 - \frac{C}{\lambda_1} |DDu|^2.
	\end{aligned}
\end{equation} 
By the commutation formulae \eqref{order}, we observe that
\[
g_{1 \ov \ell i} - g_{i \ov \ell 1} = - T_{i1}^p u_{p \ov \ell} + H_{i \ov \ell 1},
\]
where
\[
H_{i \ov \ell 1} = D_i (\chi_{1 \ov \ell } + a_1 u_{\ov \ell} + a_{\ov \ell} u_1) - D_1 (\chi_{i \ov \ell } + a_i u_{\ov \ell } + a_{\ov \ell} u_i).
\]
Since $u_{1i} - u_{i1} = - T_{1i}^p u_p \leq C$,
we have 
\[
|H_{i \ov \ell 1}| \leq C + C \lambda_1.
\]
So, for any $0 < \epsilon < 1$,
\begin{equation}
	\label{H-3rd}
	|g_{1 \ov \ell i}|^2 \leq
	(1 + \epsilon) |g_{i \ov \ell 1}|^2 + \Big(\frac{1}{\epsilon} + 1\Big) \lambda_1^2.
\end{equation}
From \eqref{f-1}, we see that $g_{i\ov 1 1} = 0$ for $1 < i \leq \mu$.
Hence, by\eqref{H-3rd}, it follows that
\begin{equation}\label{g11i}
    \frac{F^{i \ov i} |g_{1 \ov 1 i}|^2}{\lambda_1^2}
    \leq \frac{F^{1 \ov 1} |g_{1 \ov 1 1}|^2}{\lambda_1^2} + ( 1 + \epsilon)
    \sum_{i > \mu } \frac{F^{i \ov i} |g_{i \ov 1 1}|^2}{\lambda_1^2}
    + \frac{C \mathcal{F}}{\epsilon}.
\end{equation}
By
\[
 \epsilon = \frac{1}{16\phi - 3} \; \mbox{and}\;
- F^{p \overline{p}, q \overline{q}} 
= \frac{F^{q \overline{q}} - F^{p \overline{p} } }{\lambda_q - \lambda_p},
\]
we have
\[
	\begin{aligned}
&\; \sum_{i > 1} F^{1 \ov 1, i \ov i} |g_{i \ov 1 1}|^2 + \sum_{\ell > \mu} \frac{ F^{1 \ov 1} |g_{ \ell \ov 1 1}|^2 }{ \lambda_1 - \lambda_\ell } - (1 + \epsilon) \Big(1 - \frac{3}{16 \phi}\Big) 
\sum_{i > \mu} \frac{F^{i \ov i} |g_{i \ov 1 1}|^2}{\lambda_1} \\
\geq &\; \sum_{i > \mu} \Big(F^{1 \ov 1, i \ov i} + \frac{ F^{1 \ov 1} }{ \lambda_1 - \lambda_i } \Big) |g_{i \ov 1 1}|^2 
- \Big(1 - \frac{1}{8 \phi}\Big) \sum_{i > \mu} \frac{F^{i \ov i} |g_{i \ov 1 1}|^2}{\lambda_1} \\
= &\; \sum_{i > \mu} \Big( \frac{1}{\lambda_1 - \lambda_i} - (1 - \frac{1}{8 \phi} ) \frac{1}{\lambda_1} \Big) F^{i \ov i} |g_{i \ov 1 1}|^2
\geq 0.
\end{aligned}
\]
Substituting \eqref{g11i} into \eqref{G4} and by the above inequality, we obtain the following,
\begin{equation}
	\label{G5}
	\begin{aligned}
		0 \geq &\; - \frac{ F^{p \overline{p}, q \overline{q}} g_{p \overline{p} 1} g_{q \overline{q} \ov 1} }{\lambda_1} + 2 \frac{|\sum_i F^{ii} g_{i \ov i 1}|^2}{ \lambda_1 F} + \Big( 1 -  \frac{1}{8 \phi} \Big) \sum_{i > \mu} \frac{F^{i \ov i} |g_{1 \ov i i}|^2}{\lambda_1^2}\\
		&\; - \Big(1 - \frac{3}{16 \phi}\Big) \frac{F^{1 \ov 1} |g_{1 \ov 1 1}|^2}{\lambda_1^2} + \frac{\varphi'}{8 C\lambda_1} ( |DDu|^{2} + |D \ov{D}u|^{2} ) \\
		&\; - \varepsilon \phi' \mathcal{F} - C ( \varphi' + \phi ) \mathcal{F} - C (\varphi' - \phi' +  \lambda_1) - \frac{C }{\lambda_1} |DDu|^2.
	\end{aligned}
\end{equation} 
Similarly to \eqref{H-3rd}, by the Cauchy-Schwarz inequality, we observe that
\begin{equation}
\label{g1ii}
    |g_{1\ov i i}|^2 \geq (1 - \epsilon') |g_{i \ov i 1}|^2 - \frac{C}{\epsilon'} \la_1^2. 
\end{equation}
From \eqref{G5} and \eqref{g1ii}, 
we derive the following,
\begin{equation}
	\label{G6}
	\begin{aligned}
		0 \geq &\; - \frac{ F^{p \overline{p}, q \overline{q}} g_{p \overline{p} 1} g_{q \overline{q} \ov 1} }{\lambda_1} + 2 \frac{|\sum_i F^{ii} g_{i \ov i 1}|^2}{ \lambda_1 F} 
		- \Big(1 - \frac{3}{16 \phi}\Big) \frac{F^{1 \ov 1} |g_{1 \ov 1 1}|^2}{\lambda_1^2}\\
		&\; + (1 - \epsilon') \Big( 1 -  \frac{1}{8 \phi} \Big) 
        \sum_{i > \mu} \frac{F^{i \ov i} |g_{i \ov i 1}|^2}{\lambda_1^2}
		 + \frac{\varphi'}{8 C\lambda_1} ( |DDu|^{2} + |D \ov{D}u|^{2} ) \\
		&\; - \varepsilon \phi' \mathcal{F} - C \Big( \varphi' + \phi + \frac{1}{\epsilon'}\Big) \mathcal{F} - C (\varphi' - \phi' +  \lambda_1) - \frac{C }{\lambda_1} |DDu|^2.
	\end{aligned}
\end{equation} 
Take
\[
 \epsilon' = \frac{1}{32 \phi - 5} \; \mbox{and}\; 
 \gamma = \frac{1}{32\phi - 4}.
\]
Then note that
\[
1 - \frac{3}{16\phi} = (1 - \gamma)\Big(1- \frac{1}{8\phi}\Big) (1 - \epsilon').
\]
According to $g_{1\ov i i} = 0$ for $1 < i \leq \mu$,
it follows that $|g_{i \ov i 1}| \leq C \lambda_1^2$.
By Lemma \ref{key-c}, we have
\[
- \frac{F^{p \overline{p}, q \overline{q}} g_{p \overline{p} 1} g_{q \overline{q} \ov 1} }{\lambda_1} + 2 \frac{(\sum_i F^{ii} g_{i \ov i 1})^2}{ \lambda_1 F} +  \sum_{i > \mu} \frac{F^{i \ov i} |g_{i \ov i 1}|^2}{\lambda_1^2}
\geq  (1 - \gamma) \frac{F^{1 \ov 1} |g_{1 \ov 1 1}|^2}{\lambda_1^2} - C\mathcal{F},
\]
when $\delta \ll 1$.
Take $K \gg N \gg 1$ sufficiently large so that
\[
- \varepsilon \phi' \mathcal{F} - C ( \varphi' + \phi ) \mathcal{F} > 0
\]
and
\[
 \frac{\varphi'}{16 C\lambda_1} ( |DDu|^{2} + |D \ov{D}u|^{2} )
 \geq C \lambda_1 + \frac{C}{\lambda_1} |DDu|^2.
\]
Note that
\[
- F^{p \overline{p}, q \overline{q}} g_{p \overline{p} 1} g_{q \overline{q} \ov 1} +  2 \frac{|\sum_i F^{ii} g_{i \ov i 1}|^2}{ F}
> 0.
\]
By the above four inequalities, 
we finally derive from \eqref{G6} that
\begin{equation}
	\label{G7}
	\begin{aligned}
		0\geq \frac{N \varphi}{16 C} \lambda_1  - C ( \varphi' - \phi'),
	\end{aligned}
\end{equation} 
from which we derive an upper bound for $\lambda_1$. This completes the proof of Theorem \ref{thm1}.


\begin{thebibliography}{99}
	
\bibitem{J.M.B} J.M. Ball, Differentiability properties of symmetric and isotropic functions, Duke Math. J. 51 (1984), 699-728.

\bibitem{CNS} L. Caffarelli, L. Nirenberg and J. Spruck, Dirichlet problem for nonlinear second order elliptic equations III, Functions of the eigenvalues of the Hessian, Acta Math. 155 (1985), 261-301.

\bibitem{ChY} S.Y. Cheng and S.-T. Yau, On the existence of a complete K\"ahler metric on noncompact complex manifolds and the regularity of Fefferman's equation, Comm. Pure Appl. Math., 33 (1980), 507-544. 

\bibitem{Chu} J.C. Chu, A simple proof of curvature estimate for convex solution of $k$-Hessian equation, Proc. Amer. Math. Soc. 149 (2021), no. 8, 3541–3552.



\bibitem{CHZ} J.C. Chu, L.D. Huang and X.H. Zhu, The Fu-Yau equation in higher dimensions, Peking Math. J. 2 (2019), no. 1, 71–97.


\bibitem{CHZ3} J.C. Chu, L.D. Huang and X.H. Zhu, The second Hessian type equation on almost Hermitian manifolds, Front. Math. 19 (2024), no. 6, 961–988. 

\bibitem{CP} T. Collins and S. Picard, The Dirichlet problem for the $k$-Hessian equation on a complex manifold, Amer. J. Math. 144 (2022), no. 6, 1641–1680.

\bibitem{CJY}  T.C. Collins, A. Jacob and S.-T. Yau, $(1, 1)$ forms with specified Lagrangian phase: A priori estimates and algebraic obstructions, Camb. J. Math. 8 (2020), no. 2, 407-452.

\bibitem{CS} T.C. Collins and G. Sz\'ekelyhidi, Convergence of the J-flow on toric manifolds. J. Differential Geom. 107 (2017), no. 1, 47–81.


\bibitem{CY} T.C. Collins and S.-T. Yau, Moment maps, nonlinear PDE, and stability in mirror symmetry, I: geodesics, Ann. PDE 7 (2021), no. 1, Paper No. 11, 73 pp.


\bibitem{DK} S. Dinew and S. Ko\l odziej, Liouville and Calabi-Yau type theorems for complex Hessian equations, Amer. J. Math. 139 (2017), 403-415.

\bibitem{Dong1} W.S. Dong, Second order estimates for a class of complex Hessian equations on Hermitian manifolds. J. Funct. Anal. 281 (2021), no. 7, Paper No. 109121, 25 pp.

\bibitem{Dong} W.S. Dong, Second order estimates for complex Hessian equations with gradient terms on both sides, J. Differential Equations 271 (2021), 128–151. 

\bibitem{DL} W.S. Dong and C. Li, Second order estimates for complex Hessian equations on Hermitian manifolds, Discrete Contin. Dyn. Syst. 41 (2021), no. 6, 2619–2633.

\bibitem{DXZ} W.S. Dong, S.R. Xu and R.J. Zhang, Pogorelov interior estimates for general sum-type Hessian equations, arXiv:2603.15345.

\bibitem{FLM} H. Fang, M.J. Lai and X.-N. Ma,
On a class of fully nonlinear flows in K\"ahler geometry,
J. Reine Angew. Math. 653 (2011), 189-220.


\bibitem{FinoLi} A. Fino, Y.Y. Li, S. Salamon and L. Vezzoni, The Calabi-Yau equation on 4-manifolds over 2-tori, Trans. Amer. Math. Soc. 365 (2013), no. 3, 1551-1575.

\bibitem{F-Y1} J.X. Fu and S.-T. Yau, A Monge-Amp\`ere-type equation motivated by string theory, Comm. Anal. Geom. 15 (2007), 29-75.

\bibitem{F-Y2} J.X. Fu and S.-T. Yau, The theory of superstring with flux on non-K\"ahler manifolds and the complex Monge-Amp\`ere equation, J. Differential Geom. 78 (2008), 369-428.


 

\bibitem{GL} B. Guan and Q. Li, The Dirichlet problem for a complex Monge-Amp\`ere type equation on Hermitian manifolds, Adv. Math. 246 (2013), 351-367.

\bibitem{GN} B. Guan and X.L. Nie, Fully nonlinear elliptic equations with gradient terms on Hermitian manifolds,
Int. Math. Res. Not. IMRN 2023, no. 16, 14006–14042.

\bibitem{GuanSun} B. Guan and W. Sun, On a class of fully nonlinear elliptic equations on Hermitian manifolds, Calc. Var. Partial Differential Equations 54 (2015), no. 1, 901-916.



\bibitem{GRW} P.F. Guan, C.Y. Ren and Z.Z. Wang, Global $C^2$-estimates for convex solutions of curvature equations, Comm. Pure Appl. Math. 68 (2015), 1287-1325.

\bibitem{GS} B. Guo and J. Song, Sup-slopes and sub-solutions for fully nonlinear elliptic equations, arXiv:2405.03074.


\bibitem{HL} R. Harvey, H.B. Lawson, Calibrated geometries. Acta. Math., 148 (1982), 47-157.


\bibitem{HMW} Z.L. Hou, X.-N. Ma and D.M. Wu, A second order estimate for complex Hessian equations on a compact K\"ahler manifold, Math. Res. Lett. 17 (2010), 547-561.  

\bibitem{HS} G. Huisken and C. Sinestrari, Convexity estimates for mean curvature flow and singularities of mean convex surfaces, Acta. math. 183 (1999), 45-70.


\bibitem{K} N. Korevaar, A priori interior gradient bounds for solutions to elliptic Weingarten equations, Ann. Inst. H. Poincar\'e Anal. Non Lin\'eaire 4 (1987), no. 5, 405–421.



\bibitem{LYZ} C. Leung, S.-T. Yau, and E. Zaslow, From special Lagrangian to Hermitian-Yang-Mills via Fourier-Mukai transform, Winter School on Mirror Symmetry, Vector Bundles and Lagrangian Submanifolds, (1999), 209-225, AMS. IP Stud. Adv. Math., 23, Amer, Math. Soc., Providence, RI, 2001.


\bibitem{LRW} M. Li, C.Y. Ren and Z.Z. Wang,
An interior estimate for convex solutions and a rigidity theorem, J. Funct. Anal. 270 (2016), 2691-2714.


\bibitem{LiSY} S.Y. Li, On the Dirichlet problems for symmetric function equations of the eigenvalues of the complex Hessian, Asian J. Math. 8 (2004), 87-106.

\bibitem{YYLi} Y.Y. Li, Some existence results of fully nonlinear elliptic equations of Monge-Ampere type, Comm. Pure Appl. Math. 43 (1990), 233-271.

\bibitem{LTr}
 M. Lin and N.S. Trudinger, On some inequalities for elementary symmetric functions, Bull. Austral. Math. Soc. 50:2 (1994), 317–326.
 
\bibitem{LT} S.Y. Lu and Y.-L. Tsai, A simple proof of curvature estimates for the $n-1$ Hessian equation, Proc. Amer. Math. Soc. 154 (2026), no. 2, 893–904.

\bibitem{PPZ} D.H. Phong, S. Picard and X.W. Zhang, A second order estimate for general complex Hessian equations, Anal. PDE 9 (2016), 1693-1709.

\bibitem{P-P-Z3} D.H. Phong, S. Picard and X.W. Zhang, Fu-Yau Hessian equations, J. Differential Geom. 118 (2021), no. 1, 147–187.


\bibitem{RW1} C.Y. Ren and Z.Z. Wang, On the curvature estimates for Hessian equations, Amer. J. Math. 141 (2019), 1281-1315.

\bibitem{RW2} C.Y. Ren and Z.Z. Wang, The global curvature estimate for the $n-2$ Hessian equation, Calc. Var. Partial Differential Equations 62 (2023), no. 9, Paper No. 239, 50 pp.


\bibitem{SY}  R. Shankar and Y. Yuan, Hessian estimates for the sigma-2 equation in dimension four, Ann. of Math. (2), 201 (2025), no. 2, 489--513.

\bibitem{SW} J. Song and B. Weinkove,
On the convergence and singularities of the $J$-flow with applications to the Mabuchi energy, Comm. Pure Appl. Math. 61 (2008), no. 2, 210-229.

\bibitem{G.S} G. Sz\'ekelyhidi, Fully non-linear elliptic equations on compact Hermitian manifolds, J. Differential Geom. 109 (2018), 337-378.

\bibitem{STW} G. Sz\'ekelyhidi, V. Tosatti and B. Weinkove, Gauduchon metrics with prescribed volume form, Acta Math. 219 (2017), 181-211.





\bibitem{TY1} G. Tian and S.-T. Yau, Complete K\"ahler manifolds with zero Ricci curvature. I, J. Amer. Math. Soc., 3 (1990), 579-609. 

\bibitem{TY2} G. Tian and S.-T. Yau, Complete K\"ahler manifolds with zero Ricci curvature. II, Invent. Math., 106 (1991), 27-60.

\bibitem{TW} V. Tosatti and B. Weinkove, The complex Monge-Amp\`ere equation on compact Hermitian manifolds, J. Amer. Math. Soc., 23 (2010), 1187-1195.

\bibitem{T-W3} V. Tosatti and B. Weinkove, Hermitian metrics, $(n-1,n-1)$ forms and Monge-Amp\`ere equations,
J. Reine Angew. Math. 755 (2019), 67-101.

\bibitem{T-W19} V. Tosatti and B. Weinkove, The complex Monge-Amp\`ere equation with a gradient term, Pure Appl. Math. Q. 17 (2021), no. 3, 1005–1024.


\bibitem{Wang} X.-J. Wang, The $k$-Hessian equation, Geometric analysis and PDEs, Lecture Notes in Math., vol. 1977, Springer, Dordrecht, 2009, pp. 177–252.

\bibitem{Yau} S.-T. Yau, On the Ricci curvature of a compact Kähler manifold and the complex Monge-Amp\`e re equation. I, Comm. Pure Appl. Math., 31 (1978), 339-411.

\bibitem{Yuan} R.R. Yuan, On a class of fully nonlinear elliptic equations containing gradient terms on compact Hermitian manifolds, Canad. J. Math. 70 (2018), 943-960.


\bibitem{ZhangDK} D.K. Zhang, Hessian equations on closed Hermitian manifolds, Pacific J. Math. 291 (2017), 485-510.

\bibitem{Zhang} R.J. Zhang, $C^2$ estimates for $k$-Hessian equations and a rigidity theorem, Adv. Math. 480 (2025), part A, Paper No. 110488.



\end{thebibliography}
\end{document}